\documentclass[10pt,a4paper]{article}
\usepackage[latin1]{inputenc}
\usepackage{amsmath}
\usepackage{amsfonts}
\usepackage{amssymb}
\usepackage{graphicx}
\usepackage{authblk}
\usepackage{fullpage}
\usepackage{xcolor}

\usepackage{subcaption}
\usepackage{pgfplots}

\usepackage{multirow}

\usepackage{amsthm}

\usetikzlibrary{calc,patterns,plotmarks,decorations.pathmorphing,	decorations.markings,arrows}

\newtheorem{theorem}{\bf Theorem}[section]

\newtheorem{fact}{\bf Fact}[section]

\theoremstyle{definition}

\theoremstyle{remark}
\newtheorem{remark}{\bf Remark }[section]
\theoremstyle{definition}
\newtheorem{example}{\bf Example }[section]

\DeclareMathOperator{\Tr}{Tr}
\DeclareMathOperator{\sign}{sign}

\usepackage{enumitem}
\newcommand{\subscript}[2]{$#1   #2$}

\begin{document}
		\title{When does a  Periodic Response Exist in a Periodically Forced Multi-Degree-of-Freedom  Mechanical System?}

	\author{Thomas Breunung\footnote{Corresponding author, E-mail: brethoma@ethz.ch, Phone: +41 44 633 83 56  }~ and George Haller  }
	\affil{Institute for Mechanical Systems, ETH Z{\"u}rich, \protect\\  Leonardstrasse 21, 8092 Z{\"u}rich, Switzerland}
	\maketitle

	

	\begin{abstract}
While periodic responses of periodically forced dissipative nonlinear mechanical systems are commonly observed in experiments and numerics, their existence can rarely be concluded in rigorous mathematical terms. This lack of a priori existence criteria for mechanical systems hinders definitive conclusions about periodic orbits from approximate numerical methods, such as harmonic balance. In this work, we establish results guaranteeing the existence of  a periodic response  without restricting the amplitude of the forcing or the response. Our results provide a priori justification for the use of numerical methods for the detection of periodic responses. We illustrate on examples that each condition of the existence criterion we discuss is essential. 

\vspace{0.2cm}
	\end{abstract}

\section{Introduction}

Nonlinear mechanical systems are generally assumed to approach a periodic orbit under external periodic forcing. While approximately periodic responses are commonly observed in numerical routines (eg. numerical time integration, numerical continuation or harmonic balance) and experiments, concluding the existence of periodic response rigorously in a nonlinear system is more delicate.

For nonlinear system close to a solvable limit, perturbation methods remain a powerful tool to compute  periodic responses. Among these, the method of averaging requires slowly varying amplitude equations (cf. Sanders et al.~\cite{Verhulst_Avg}), while the method of multiple scales (cf. Nayfeh~\cite{Nayfeh_Perturb}) assumes evolution on different time scales generated by small parameters. The method of normal forms (cf. Murdock~\cite{Murdock_NF}) introduces a series of smooth coordinate changes to approximate and simplify the essential dynamics in a Taylor series in a small enough neighborhood of an equilibrium. Due to the truncation of infinite series arising in these procedures, the approximate dynamics remains valid only for sufficiently small values of an underlying perturbation parameter. How small that parameter is required to be is generally unclear, and hence the relevance of the results obtained from perturbation procedures under physically relevant parameter values is a priori unknown.

Rigorous numerics (cf. van den Berg and Lessard~\cite{Berg_RigorousNumerics}), estimating the ignored tail of Taylor expansions, is only applicable to specific numerical examples. In applications, therefore, one typically employs an additional numerical method to verify the predictions of perturbation methods. This is clearly not optimal.  

Due to the broad availability of effective numerical packages, numerical time-integration is often used to compute  periodic responses.  Modern, lightly damped engineering structures, however, require long integration times to reach a steady state response. Furthermore, the observed  periodic response depends on the initial condition and unstable branches of the full set of periodic responses cannot be located in this fashion. 

More advanced numerical schemes, such as harmonic balance (cf. Mickens~\cite{Mickens_HB}) and numerical collocation (cf. Ascher et al.~\cite{Ascher_collo}), reformulate the underlying ordinary differential equations or boundary value problem into a set of algebraic equations by approximating the  periodic response in terms of a set of finite basis functions (e.g., polynomial or Fourier basis). Coupled with numerical continuation schemes, this approach enables the computation of periodic responses even for high-amplitude oscillations (cf. Dankowicz and Schilder~\cite{coco}). To justify this procedure a priori and estimate the error due to the truncation of the infinite-dimensional basis-function space, the existence of the periodic orbit would need to be established by an analytic criterion. While this can be guaranteed under small forcing or small nonlinearities by Poincar\'{e} map arguments, the problem remains unsolved for general, forced nonlinear mechanical systems.

Recently, Jain et al.~\cite{NPO} have proposed iterative methods to efficiently compute  periodic responses of periodically forced nonlinear mechanical systems without small parameter assumptions. Their existence criterion, derived via the Banach fixed-point theorem, however, fails for high forcing amplitudes and forcing frequencies in resonance with an eigenfrequency  of the linearized system. 

In the absence of small parameters, general fixed point theorems, such as Brouwer's or Schauder's fixed point theorem or the Leray-Schauder principle  are powerful tools to prove the existence of periodic orbits (cf. Bobylev et al.~\cite{Bobylev} or Precup~\cite{Precup} for a summary). Lefschetz~\cite{lefschetz1943existence},  for example, proved the existence of a periodic orbit for a specific one-degree-of-freedom, forced nonlinear mechanical system. His result, however, requires the damping force to be of the same order as the geometric nonlinearities. Therefore, his argument does not apply to common mechanical systems, such as the classic Duffing oscillator, with linear damping. 

This restriction on the damping was relaxed significantly by Lazer~\cite{Lazer_Schauder}.  Based on Schauder's fixed point theorem, Lazer's result allows for linear damping but requires a growth restriction on the nonlinearity. This result was further strengthened and extended to higher dimensions by Mawhin~\cite{Mawhin_LazerExt}, who  required the damping simply to be differentiable. Both results, however,  restrict the growth of the nonlinearties to be less than linear for sufficiently high displacements, i.e., exclude polynomial or even linear stiffness forces. As Martelli~\cite{Martelli} noted, this restriction can be relaxed to a linear growth with sufficiently low slope, depending on the eigenfrequency of the system and the forcing frequency. Due to the growth restriction on the nonlinearities these results are inapplicable for simple polynomial nonlinearities.  

The results of Mawhin and Lazer have been extended to nonsmooth systems  by Chu et al.~\cite{Chu_singular} and Torres~\cite{torres2003existence} and  to more complex differential operators (cf. Mawhin~\cite{Mawhin2001}),  relying on an extension of the generalized continuation theorems by Gaines and Mawhin~\cite{Mahwin_Cidx} and Man\'{a}sevich and Mawhin~\cite{mawhin_pLaplacian}. Furthermore,  Antman and Lacabonara~\cite{Antman2009} give an existence criterion for periodic solutions based on the principle of guiding functions (cf. Krasnosel'skij~\cite{Krasnos_guidingFunc}). This result, however, relies on the specific form of the nonlinearity and external forcing for shells.

A general existence criterion for periodic orbits in second-order differential equations with linear dissipation  can be found in the work of  Rouche and Mawhin~\cite{Rouche_PO}. Their result implies  the existence of a periodic response for dissipative nonlinear mechanical systems for arbitrary large forcing amplitudes. It appears, however, that the results in~\cite{Rouche_PO} are not known in the mechanical vibrations literature. 

In this paper, we refine the Rouche-Mawhin results to be directly relevant for mechanical systems. This gives a general sufficient criterion for the existence of periodic orbits in multi-degree-of-freedom forced-damped nonlinear mechanical systems, without any restriction on the magnitude of the forcing or vibration amplitude.  We also give mechanically relevant examples of periodically forced systems violating our criterion in which a periodic response does not exist. These show that the assumptions in our results are indeed relevant for mechanical systems and cannot be individually omitted without loosing the conclusion. Further, we identify a broad family  of nonlinear mechanical systems for which our theorem guarantees the existence of a periodic response. This result enables the rigorous computation of periodic orbits for a large class of strongly nonlinear mechanical systems.

\section{Problem statement}

We consider a general~$N$-degree-of-freedom mechanical system of the form
\begin{equation}
\mathbf{M}\ddot{\mathbf{q}}+\mathbf{C}\dot{\mathbf{q}}+\mathbf{S}(\mathbf{q})= \mathbf{f}(t), \qquad \mathbf{f}(t+T)=\mathbf{f}(t) ,\qquad \mathbf{q}\in \mathbb{R}^N,
\label{eq:sys0}
\end{equation}
where the mass matrix~$\mathbf{M}\in \mathbb{R}^{N\times N}$ is positive definite and~$\mathbf{C}\in \mathbb{R}^{N\times N}$ denotes the damping matrix. The geometric nonlinearities~$\mathbf{S}(\mathbf{q})$ contain all position-dependent forces, including potential and non-potential forces, such as follower forces~\cite{Bolotin}. We also assume that~$\mathbf{S}(\mathbf{q})$ to be  continuous in its arguments.  The external forcing~$\mathbf{f}$ is assumed to be~$T$-periodic and continuous.   

To emphasize the importance of rigorous existence criteria for periodic orbits of nonlinear mechanical systems, we next present two examples for which the popular harmonic balance procedure leads to  wrong conclusions. First, addressing the  common belief that periodic forcing always leads to a periodic response of system~\eqref{eq:sys0}, we start with an example illustrating the contrary. Furthermore, we demonstrate numerically that the harmonic balance method yields false results on this example. Secondly, we demonstrate that for a linear system the harmonic balance procedure can  predict an inaccurate periodic orbit.

\subsection{Motivating examples}
\label{sec:counter_ex1}

In the harmonic balance procedure, the equation of motion of the dynamical system~\eqref{eq:sys0} is evaluated along an assumed periodic orbit of the form 
\begin{equation}
	\mathbf{q}^*(t)=\frac{\mathbf{c}_0}{2}+
	\sum_{k=1}^{K} 
	\left[
	\mathbf{s}_k\sin(k\Omega t)+\mathbf{c}_k\cos(k\Omega t)
	\right],\qquad
	 \mathbf{s}_k,\mathbf{c}_k\in\mathbb{R}^N.
	\label{eq:HB_sol}
\end{equation}
Next, the forcing term~$\mathbf{f}(t)$ and the nonlinearity evaluated along the postulated periodic orbit,~$\mathbf{S}(\mathbf{q}^*(t))$, are projected on the first~$K$ Fourier modes and higher modes are ignored. As a result, one obtains a finite set of nonlinear algebraic equations for the unknown constants~$\mathbf{c}_k$ and~$\mathbf{s}_k$. This set of equations can generally not be solved analytically, therefore iterative methods, such as a Newton-Raphson iteration, are employed to generate an approximate solution.  For more details, we refer to Mickens~\cite{Mickens_HB}. 

Besides the a priori assumption of the existence of a periodic orbit, the truncation of the periodic orbit~\eqref{eq:HB_sol} at some finite order~$K$ needs to be justified. Classic results (cf. Bobylev et al.~\cite{Bobylev} or Leipholz~\cite{leipholz1977direct}) show that this truncation can be justified for a sufficiently large~$K$ when a periodic orbit of system~\eqref{eq:sys0} actually exists. If the existence of a periodic orbit cannot  be guaranteed a priori, conditions derived by Urabe~\cite{urabe} or Stokes~\cite{stokes} might guarantee the existence of a periodic orbit close to the harmonic balance approximation. These conditions, however, can only be evaluated a posteriori, as they rely on the harmonic balance solution itself. Notably, Kogelbauer et al.~\cite{Kogelbauer_HB} strengthen the results of Urabe~\cite{urabe} and Stokes~\cite{stokes}, and provide an explicitly verifiable condition for the existence of a periodic orbit. In practice, however, their conditions restrict the forcing and response amplitudes to small values. 

To illustrate issues that can arise with harmonic balance, we consider the  two degree-of-freedom oscillator 
\begin{equation}
	\begin{bmatrix}
		m_1 & 0\\
		0 & m_2
	\end{bmatrix}
	\ddot{\mathbf{q}}
	+
	\begin{bmatrix}
		c_1+c_2 & -c_2\\
		-c_2 &c_1+c_2
	\end{bmatrix}
	\dot{\mathbf{q}}
	+
	\begin{bmatrix}
		k_1+k_2 & -k_2\\
		-k_2 &k_1+k_2
	\end{bmatrix}
	\mathbf{q}
	+\kappa
	\begin{bmatrix}
		q_1^2+q_2^2 \\
		q_1^2+q_2^2 
	\end{bmatrix}
	=
	\begin{bmatrix}
		f_1 \\
		f_2 
	\end{bmatrix}.
	\label{eq:eqm_counter1}
\end{equation}
The nonlinearities assumed in system~\eqref{eq:eqm_counter1} may arise in a Taylor series approximation of a more complex nonlinear forcing vector, which is terminated at second order. Specifically, quadratic nonlinearities  arise in the modeling of ship capsize~(cf. Thompson~\cite{thompson_cap}),  ear drums (cf. Mickens~\cite{Mickens_Intro_NL}) and shells (cf. Antman and Lacabonara~\cite{Antman2009}). Touz\'{e} et al.~\cite{touze_springsys_cons} study a spring-mass system in which quadratic nonlinearities arise naturally due to the geometry. 

We assume forcing in the form of a triangular wave
\begin{equation}
	f_1(t)=-f_2(t)=\frac{2 f_m}{\pi}\int_0^t\sign(\cos(\Omega s))ds=f_m\frac{8}{\pi^2}\sum_{k=0}^{\infty}(-1)^k\frac{\sin((2k+1)\Omega t)}{(2k+1)^2},
	\label{eq:trigwave}
\end{equation}
where the parameter~$f_m$ denotes the amplitude and~$\Omega=2\pi/T$ the excitation frequency. For the remaining parameters, we assume the following non-dimensional numerical values
\begin{equation}
	m_1=m_2=1,\quad k_1=1,\quad k_2=4,\quad c_1=0.001,\quad c_2=0,\quad f_m=0.01178,\quad \Omega=1,\quad \kappa=1.
	\label{eq:pars_counter1}
\end{equation}
We apply the harmonic balance to system~\eqref{eq:eqm_counter1} with the parameters~\eqref{eq:pars_counter1} and forcing~\eqref{eq:trigwave}. We solve the resulting algebraic system of equations with a Newton-Raphson iteration, whereby we evaluate the nonlinearity in the time domain and transform the time signal to the frequency domain using fast Fourier transforms (cf. Cameron and Griffin~\cite{cameron_AFT}). Following  common practice (cf. Cochelin and Vergez~\cite{cochelin2009high}), we start with a low number of harmonics and successively increase the number $K$ of harmonics in ansatz~\eqref{eq:HB_sol} until the resulting oscillation amplitude appears converged, i.e., does not change with increasing~$K$. We depict the result of this harmonic balance procedure in Fig.~\ref{fig:C1}. 
\begin{figure}[ht!]
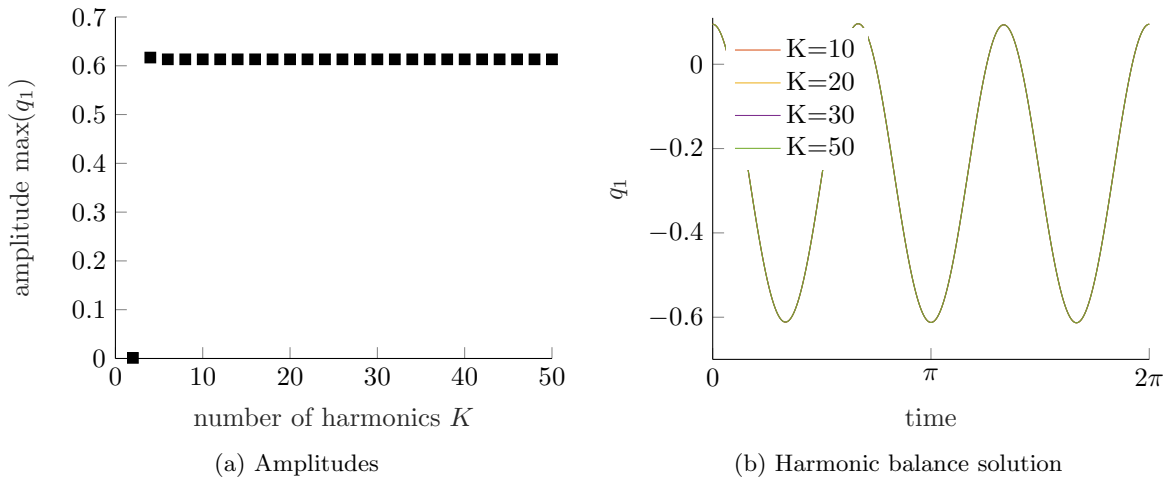

	\begin{center}
		\begin{subfigure}[b]{0.49\textwidth}
%
			
			\caption{Amplitudes}
			\label{fig:C1_apmls}
		\end{subfigure}
		\begin{subfigure}[b]{0.49\textwidth}
			
			\definecolor{mycolor1}{rgb}{0.00000,0.44700,0.74100}%
			\definecolor{mycolor2}{rgb}{0.85000,0.32500,0.09800}%
			\definecolor{mycolor3}{rgb}{0.92900,0.69400,0.12500}%
			\definecolor{mycolor4}{rgb}{0.49400,0.18400,0.55600}%
			\definecolor{mycolor5}{rgb}{0.46600,0.67400,0.18800}%
			%
%
			\caption{Harmonic balance solution}
			\label{fig:C1_tseries}
		\end{subfigure}
	\end{center}
	\caption{Amplitudes and  time series obtained with the harmonic balance procedure for the mechanical system~\eqref{eq:eqm_counter1} with parameters~\eqref{eq:pars_counter1} and forcing~\eqref{eq:trigwave}.}
	\label{fig:C1}
\end{figure}

From the amplitudes depicted in Fig.~\ref{fig:C1_apmls}, one would normally conclude the converge of the displacement amplitude of the first coordinate to a value of about~$0.6$. Also, the time series for the choices of $K$, shown in Fig.~\ref{fig:C1_tseries}, are practically indistinguishable.  There is,  therefore, every indication that the harmonic balance method has correctly identified a periodic orbit for system~\eqref{eq:eqm_counter1}.

The periodic orbit suggested by the numerical result in Fig.~\ref{fig:C1}, however, does \textit{not exist} in system~\eqref{eq:eqm_counter1} for the  parameters~\eqref{eq:pars_counter1} and the forcing~\eqref{eq:trigwave}. More generally, in Appendix~\ref{app:C1} we prove that if the amplitude $f_m$ of the forcing~\eqref{eq:trigwave} is above a certain threshold, no periodic orbit exists for system~\eqref{eq:eqm_counter1}. Since the value of the forcing amplitude~\eqref{eq:pars_counter1} is above this threshold, the periodic orbit indicated by the harmonic balance procedure in Fig.~\ref{fig:C1} does not actually exists.

One might argue that due to the discontinuity of the forcing~\eqref{eq:trigwave}, the assumption of a twice differentiable solution is not justified. Indeed, due to the Lipschitz continuity of the forcing~\eqref{eq:trigwave} just the existence and uniqueness of a local solution can be guaranteed by Picard's theorem (cf. Coddington~and Levinson~\cite{Coddington}). Our nonexistence proof, however, relies only on the fact that the amplitudes of the forcing is above a certain threshold. One can therefore easily replace the forcing~\eqref{eq:trigwave} by a smoother, even analytic alternative and obtain the same conclusion. 

The unforced limit of system~\eqref{eq:eqm_counter1} has two fixed points which are connected through a homoclinic orbit. As Thompson and Steward~\cite{Thompson_NLC} observe,  these features can give rise unbounded (escape) behavior for the forced system. For small enough forcing amplitudes, the existence of the periodic orbit is guaranteed by the general results of Haro and de Llave~\cite{Haro_tori}. For larger forcing amplitudes, however, the result of Haro and de Llave~\cite{Haro_tori} cannot guarantee the existence of a periodic orbit and exceeding the threshold~\eqref{eq:f_thres} rules out the possibility of any periodic motion.

As we will see shortly, the crucial reason for the nonexistence of a periodic orbit in the above example is the form of the nonlinearity. Indeed, for a simple system with a single quadratic nonlinearity, Thompson and Steward~\cite{Thompson_NLC} were unable to continue a periodic orbit numerically for arbitrarily high forcing amplitudes. Difficulties in applying harmonics balance to systems with quadratic nonlinearities have also lead to the practical guidelines by Mickens~\cite{mickens_HBguide}, who heuristically restricts the harmonics balance procedure to systems with odd nonlinearities.

Next, we give an example for which the harmonic balance procedure yields an inaccurate periodic orbit, due its unavoidable truncation of the basis function space. We consider the linear forced-damped oscillator
\begin{equation}
\ddot{q}+c\dot{q}+kq=f(t), \qquad k=400,~~ c=0.01,
\label{eq:lin_sys}
\end{equation}
with a forcing shown in Fig.~\ref{fig:C2_forcing}. This forcing is clearly dominated by a fundamental harmonic. The harmonic balance procedure for the choice of five harmonics yields a periodic orbit, which is dominated by the fundamental harmonic with an amplitude at about $0.0025$ (c.f. Fig.~\ref{fig:C2_tseries}). Increasing the number of harmonics considered to ten and fifteen, we obtain the periodic orbits in Fig.~\ref{fig:C2_tseries}, which are practically indistinguishable from the periodic orbit obtained from five harmonics. In practice, one would practically terminate the harmonic balance procedure and conclude the convergence of method to a periodic orbit with the maximal position equal to $0.0025$ in magnitude.

\begin{figure}[ht!]
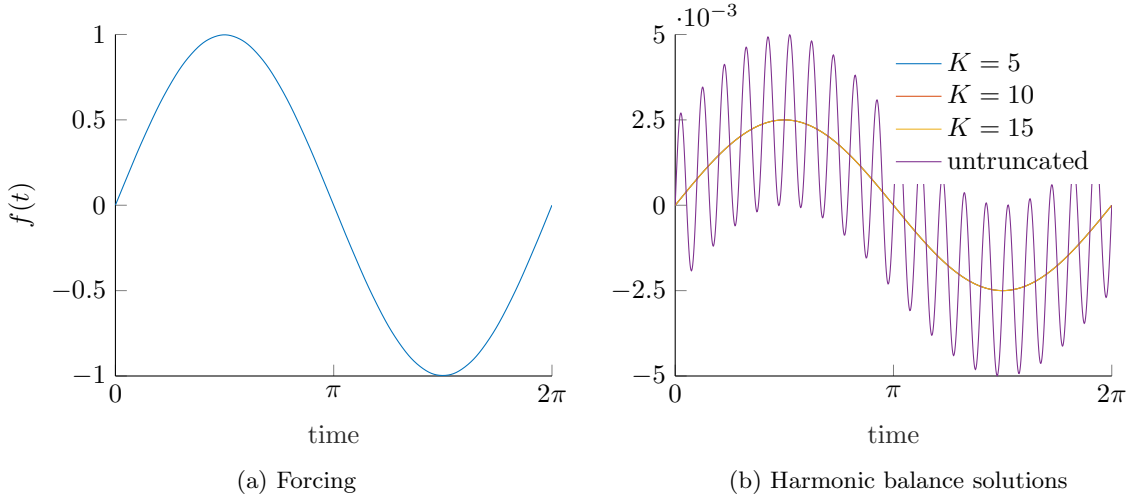

	\begin{center}
		\begin{subfigure}[b]{0.49\textwidth}
\definecolor{mycolor1}{rgb}{0.00000,0.44700,0.74100}%
%
			\caption{Forcing}
			\label{fig:C2_forcing}
		\end{subfigure}
		\begin{subfigure}[b]{0.49\textwidth}
			
\definecolor{mycolor1}{rgb}{0.00000,0.44700,0.74100}%
\definecolor{mycolor2}{rgb}{0.85000,0.32500,0.09800}%
\definecolor{mycolor3}{rgb}{0.92900,0.69400,0.12500}%
\definecolor{mycolor4}{rgb}{0.49400,0.18400,0.55600}%
%
%
			\caption{Harmonic balance solutions}
			\label{fig:C2_tseries}
		\end{subfigure}
	\end{center}
	\caption{Forcing and  periodic orbits obtained with the harmonic balance procedure for the linear oscillator~\eqref{eq:lin_sys}.}
	\label{fig:C2}
\end{figure}

The untruncated version of the periodic orbit depicted in Fig~\ref{fig:C2_tseries}, however, differs significantly from the approximation obtained from the harmonic balance procedure. The maximal position along the true periodic orbit is twice the value predicted with a low-order truncation. We note that the forcing $f(t)$ and the corresponding periodic orbit  $q^*(t)$ of system~\eqref{eq:lin_sys} are given by
\begin{equation}
 f(t)=\frac{1}{400}\left(399\sin(t)+0.01\cos(t)+0.2\cos(20t)\right),\qquad q^*(t)=0.0025(\sin(t)+\sin(20t)).
 \label{eq:lin_sol}
\end{equation}
 The periodic orbit~\eqref{eq:lin_sol} has a fundamental harmonic, which is identified correctly by the harmonic balance approximations (c.f. Fig~\ref{fig:C2_tseries}). The higher-frequency component, however,  is truncated in the harmonic balance procedure and is therefore not captured. Eventually, increasing the number of harmonics considered above twenty, we obtain the correct result from the harmonic balance, but there is no rigorous criterion that would indicate this in advance. Even in the case of infinite harmonics in ansatz~\eqref{eq:HB_sol}, the approximate periodic orbit from the harmonic balance procedure can differ significantly from the actual periodic orbit, as we demonstrate in Appendix~\ref{app:fejer}.

As we have illustrated, even in the case of an apparently convergent harmonic balance approximation the existence of a periodic orbit cannot be guaranteed.  Rigorous criteria for the existence of periodic orbits, however, can exclude false positives. 

\section{Existence of   a periodic response}

With the mean forcing defined as 
\begin{equation}
\bar{\mathbf{f}}=\frac{1}{T}\int_0^T\mathbf{f}(t)dt,
\label{eq:f_mean}
\end{equation}
 we will prove the following general result
\begin{theorem}
Assume that the forcing $\mathbf{f}(t)$ in system~\eqref{eq:sys0} is continuous and the following conditions hold:
\begin{enumerate}[label=(\subscript{C}{{\arabic*}})]
\item The damping matrix~$\mathbf{C}$  is definite, i.e., there exists a constant~$C_0\!>\!0$ such that
\begin{equation}
|\mathbf{x}^T\mathbf{C}\mathbf{x}|>C_0|\mathbf{x}|^2, \qquad \mathbf{x}\in\mathbb{R}^N.
\label{eq:damping_cond}
\end{equation}
\label{cond:damping}
\item The geometric nonlinearities derive from a potential, i.e., there exists a continuously differentiable scalar function~$V(\mathbf{q})$ such that 
\begin{equation}
\mathbf{S}(\mathbf{q})=\frac{\partial V(\mathbf{q})}{\partial \mathbf{q}}. 
\label{eq:potential}
\end{equation}
\label{cond:potential}
\item For each degree of freedom, the quantity $q_j(S_j(\mathbf{q})-\bar{f}_j)$ has a constant, nonzero sign far enough from the origin.  Specifically,  there exists a distance $r\!>\!0$ and an integer~$1\leq n \leq N$ such that
\begin{equation}
\begin{split}
q_j\left(S_j(\mathbf{q})-\bar{f}_j \right) &>0, \qquad |q_j|>r,~~j=1,...,n, 
\\
q_j\left(S_j(\mathbf{q})-\bar{f}_j \right)&<0, \qquad |q_j|>r,~~j=n+1,...,N. 
\end{split}
\label{eq:sign_cond}
\end{equation} 
\label{cond:sign_cond}
\end{enumerate}
Then system~\eqref{eq:sys0} has a twice continuously differentiable periodic orbit. 
\label{thm:Existence}
\end{theorem}
\begin{proof}
We deduce this theorem from Theorem 6.3 by Rouche and Mawhin~\cite{Rouche_PO} after the removal of an unnecessary zero-mean forcing assumption in its original version. We detail the proof in Appendix~\ref{app:Exist_proof}. 
\end{proof}

\begin{remark}
As a consequence of the proof of Theorem~\ref{thm:Existence}, we obtain an upper bound on the amplitude of the existing periodic orbit. Specifically, with the squared~$L_2$-norm
\begin{equation}
C_f^2:=\int_0^T \mathbf{f}^T\mathbf{f} dt,
\label{eq:def_CF}
\end{equation}
of the forcing an estimate for the maximal oscillation amplitude is given by
\begin{equation}
\sup_{0\leq t\leq T} |\mathbf{q}| \leq \sqrt{N}\left(r+\sqrt{T}\frac{C_f}{C_0}\right),
\label{eq:bnd_ampls}
\end{equation}
where the constant~$C_0$ is defined in equation~\eqref{eq:damping_cond}. We detail the derivation of this estimate in Appendix~\ref{app:bound_ampl}. Our bound~\eqref{eq:bnd_ampls} is stricter than that obtained by Rouche and Mawhin~\cite{Rouche_PO}, who have an additional summand of $\sqrt{T}C_f/C_0$ in equation~\eqref{eq:bnd_ampls}. Further, the bound~\eqref{eq:bnd_ampls} confirms the intuition arising from linear theory, that the maximal response amplitude is proportional to the quotient of forcing amplitude and minimum damping coefficient. The inequality~\eqref{eq:bnd_ampls} confirms this intuition for the full nonlinear system without small-parameter assumptions.

For a single degree-of-freedom harmonic oscillator ($N\!=\!1$,~$r=0$, damping coefficient~$c$ and eigenfrequency $\omega_0$) and single harmonic forcing with amplitude~$f$ at resonance, i.e., for a system
\begin{equation}
\ddot{q}+c\dot{q}+\omega_0^2q=f\sin(\omega_0 t),
\label{eq:lin_osci}
\end{equation}
 the relationship between the bound~\eqref{eq:bnd_ampls} and the exact solution~$q_{lin}\!=\!fT/(2\pi c)$ is as follows: 
\begin{equation}
\sup_{0\leq t\leq T} |\mathbf{q}| \leq \sqrt{N}\left(r+\sqrt{T}\frac{C_f}{C_0}\right)
=\frac{fT}{2c}>\frac{fT}{2\pi c}=q_{lin}.
\label{eq:bnd_lin}
\end{equation}
As expected, the bound~\eqref{eq:bnd_lin} is conservative, but only by a factor of $\pi$.

\end{remark}
 Condition~\ref{cond:sign_cond} implies a sign change of the geometric nonlinearities minus the mean forcing component-wise inside the interval $[-r, r]$. In Fig.~\ref{fig:c3_illu}, we sketch graphs of three different geometric nonlinearities. If the value of the geometric nonlinearities  $S_j(\mathbf{q})$ is greater than the mean forcing for $q_j\!<\!-r$ (i.e. lies in the upper left dotted square of Fig.~\ref{fig:c3_illu}), then the quantity $(S_j(\mathbf{q})-\bar{f}_j)q_j$ evaluated for  $q_j\!<\!-r$ is negative. Therefore, for condition~\ref{cond:sign_cond} to hold, $(S_j(\mathbf{q})-\bar{f}_j)q_j$ for $q_j\!>\!r$ needs to be negative, which implies that the geometric nonlinearity  $S_j(\mathbf{q})$ needs to be  below $\bar{f}_j$. In Fig.~\ref{fig:c3_illu}, the graph of of $S_j(\mathbf{q})$ needs to end in the lower right dotted square.  Such a nonlinearity is depicted in blue in Fig.~\ref{fig:c3_illu}. Similarly, a geometric nonlinearity satisfying  $(S_j(\mathbf{q})-\bar{f}_j)q_j\!>\!0$ for all $|q_j|\!>\!r$ needs to be contained in the two shaded regions of Fig.~\ref{fig:c3_illu}. For the red curve, we have $(S_j(\mathbf{q})-\bar{f}_j)q_j\!<\!0$ for $q_j\!<\!-r$ and $(S_j(\mathbf{q})-\bar{f}_j)q_j\!>\!0$ for $q_j\!>\!r$ , i.e. condition~~\ref{cond:sign_cond} is not satisfied.

\begin{figure}[ht!]
	\begin{center}
		\begin{tikzpicture}


		
		
		\node[pattern=north east lines, color={black!20!white} ] at (-0.5,2) (LW) [minimum width=1cm,  minimum height=2cm] {};
		\node[pattern=north east lines,  color={black!20!white} ] at (4.5,3.75) (LW) [minimum width=1cm,  minimum height=1.5cm] {};
		
		\node[pattern=dots] at (4.5,2) (LW) [minimum width=1cm,  minimum height=2cm] {};
		\node[pattern=dots] at (-0.5,3.75) (LW) [minimum width=1cm,  minimum height=1.5cm] {};
		
		\draw[thick, color={black!60!blue}]  plot [smooth] coordinates { (-0.5,3.5) (1,4) (2,1.5) (3,3.5) (4.5,2.5)};
		
		\draw[thick, color={black!60!green}]  plot [smooth] coordinates { (-0.5,2.25) (1,4.25) (2,3.5) (3,4.5) (4.5,4)};

		\draw[thick, color={black!60!red}]  plot [smooth] coordinates { (-0.5,4.25) (1,2.25) (2,1.75) (3,1.5) (4.5,4.45)};

		\draw (4,2) node[below]{\colorbox{white}{$r$}};
		\draw (4,2)-- (4,1.8) ;
		
		\draw (-0.05,2) node[below]{\colorbox{white}{$-r$}};
		\draw (-0,2)-- (-0,1.8);
		\draw[->] (-1,2) -- (5,2) node[above]{\colorbox{white}{$q_j$}};
		\draw[->] (2,0.5) -- (2,4.5)node[left]{$S_j(\mathbf{q})$};
		\draw[thick] (-1,3) -- (5,3) node[above]{\colorbox{white}{$\bar{f}_j$}};

		\end{tikzpicture}
	\end{center}
	\caption{Illustration of condition~\ref{cond:sign_cond}: The green and blue curve satisfy condition~\ref{cond:sign_cond}, while the red curve violates condition~\ref{cond:sign_cond}. }
	\label{fig:c3_illu}
\end{figure}
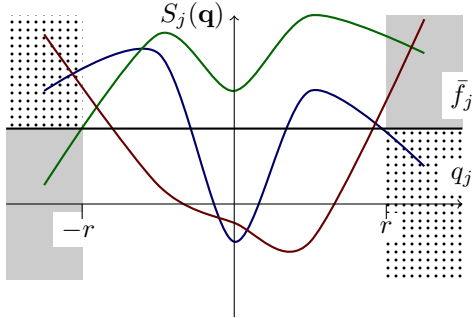

\begin{theorem}
	\label{thm:posdef}
	If the geometric nonlinearities~$\mathbf{S}(\mathbf{q})$ are differentiable, condition~\ref{cond:sign_cond} holds if
	\begin{enumerate}[label=(\subscript{C}{{\arabic* ^*}})]
		 \setcounter{enumi}{2}
		\item The Hessian of~$V(\mathbf{q})$ is definite  for~$|\mathbf{q}|\!>\!r^*$, i.e., for some constant $C_v\!>\!0$,
		\begin{equation}
		|\mathbf{x}^T\frac{\partial^2 V(\mathbf{q})}{\partial ^2 \mathbf{q}} \mathbf{x}|> C_v|\mathbf{x}|^2,\qquad \mathbf{x}\in\mathbb{R}^N,~~  |\mathbf{q}|\geq r^*.
		\end{equation}
		\label{cond:conv_cond}
	\end{enumerate} 	 
\end{theorem}
\begin{proof}
See Appendix~\ref{app:convex} for a proof.	
\end{proof}

\begin{remark}
Condition~\ref{cond:conv_cond} is more restrictive than condition~\ref{cond:sign_cond}. For example, consider the potential
	\begin{equation}
	V(q_1,q_2)=k_1 q_1^2+k_2 q_2^2,
	\label{eq:saddle_NL}
	\end{equation}
	which satisfies~\ref{cond:conv_cond} only if~$k_1$ and~$k_2$ have the same sign ($k_1k_2>0$), while it satisfies condition~\ref{cond:sign_cond} for any non-zero~$k_1$ and~$k_2$.  However, condition~\ref{cond:conv_cond} is more intuitive as it restricts the global geometry to be cup-shaped sufficiently far from the origin. In addition, condition~\ref{cond:conv_cond} is generally easier to verify, since positive or negative definiteness of the Hessian can be verified through direct eigenvalue computation, the leading minor criterion (i.e. Sylvester's criterion) or the Cholesky decomposition (cf.~Horn and Johnson~\cite{Horn_matrix}).   
\end{remark}

Theorems~\ref{thm:Existence} and~\ref{thm:posdef}  can guarantee the existence of periodic orbits for arbitrary large forcing and response amplitudes. These theorems, therefore, enable an a priori justification of the use of otherwise  heuristic approaches, such as harmonic balance or numerical collocation. In the following, we demonstrate the use of these theorems on various mechanical systems.

\section{Examples}

First, we illustrate via mechanically relevant examples that conditions~\ref{cond:damping}-\ref{cond:sign_cond} of Theorem~\ref{thm:Existence} cannot be~omitted without replacement by some other requirement. Next, we identify a large class of high-dimensional mechanical systems for which the existence of periodic orbits can be guaranteed by Theorems~\ref{thm:Existence} and~\ref{thm:posdef}. 

\subsection{The importance of condition~\ref{cond:damping}-\ref{cond:sign_cond}}

In the following we show that if one of the conditions~\ref{cond:damping}-\ref{cond:sign_cond} is violated, one can find mechanical systems with no periodic orbits that nevertheless satisfy the remaining conditions of Theorem~\ref{thm:Existence}. This underlines the importance of conditions~\ref{cond:damping}-\ref{cond:sign_cond}. 

\begin{example}[Zero-damping]
All solutions in a one-degree-of-freedom, undamped linear oscillator grow unbounded when the oscillator is forced at resonance. As a consequence, no periodic orbits may exists in such a system. Indeed, any undamped linear oscillator violates condition \ref{cond:damping}, because damping matrix is identically zero and therefore neither positive nor negative definite.  
\end{example}

\begin{example}[Non-potential nonlinearities]
We have seen in section~\ref{sec:counter_ex1} that system~\eqref{eq:eqm_counter1} has no \mbox{$T$-periodic} orbit. Indeed, trying to apply Theorem~\ref{thm:Existence} to this problem, we find that condition~\ref{cond:potential} is not satisfied.  To examine condition~\ref{cond:sign_cond}, we evaluate the quantity $S_j(\mathbf{q})-\bar{f}_j$ along $q_2=0$. This parabola opens upward and is positive outside a closed interval, i.e., 
\begin{equation}
((k_1+k_2)q_1+\kappa q_1^2-\bar{f}_1)>0, \qquad q_1>\sqrt{\frac{|\bar{f}_1|}{\kappa}},~~\text{or} ~~  q_1<-\frac{k_1+k_2}{\kappa}-\sqrt{\frac{|\bar{f}_1|}{\kappa}}.
\end{equation}
Therefore, the quantity $q_1((k_1+k_2)q_1+\kappa q_1^2-\bar{f}_1)$ is positive for all $q_1\!>\!\sqrt{\frac{|\bar{f}_1|}{\kappa}}$ and negative for $q_1\!<\!-\frac{k_1+k_2}{\kappa}-\sqrt{\frac{|\bar{f}_1|}{\kappa}}$. This implies that no constant $r$ exists such that $q_1((k_1+k_2)q_1-k_2q_2+\kappa(q_1^2+q_2^2)-\bar{f}_j)$ has a constant sign for all $|q_i|>r$, i.e. condition~\ref{cond:sign_cond} is violated. 

We now consider a slight modification of system~\eqref{eq:eqm_counter1} in the form of 
\begin{equation}
\begin{bmatrix}
1 & 0\\
0 & 1
\end{bmatrix}
\ddot{\mathbf{q}}
+
\begin{bmatrix}
c_1 & 0\\
0 & c_2
\end{bmatrix}
\dot{\mathbf{q}}
+
\begin{bmatrix}
\omega_1^2 & 0\\
0 &\omega_2^2
\end{bmatrix}
\mathbf{q}
+\kappa
\begin{bmatrix}
q_1q_2^2 \\
0
\end{bmatrix}
=
\begin{bmatrix}
f_1(t) \\
a\sin(\Omega t)
\end{bmatrix},
\label{eq:eqm_counter3}
\end{equation}
which satisfies conditions~\ref{cond:sign_cond}, for the choice
\begin{equation}
\kappa>0,~~\bar{f}_1 \leq 0,\quad \Rightarrow q_1(q_1(\omega_1^2+\kappa q_2^2)-\bar{f}_1)>0,~~~~\text{for all } q_1,q_2\in \mathbb{R}. 
\label{eq:conv_cond_eval}
\end{equation} 
Condition~\ref{cond:damping} is also satisfied for positive damping values $c_1,c_2>0$. Condition~\ref{cond:potential}, however, is not satisfied as the geometric nonlinearities of system~\eqref{eq:eqm_counter3} do not derive from a potential. In the following, we will show that system~\eqref{eq:eqm_counter3} has no $T$-periodic orbits for an appropriately chosen set of parameters. 

Assuming the contrary, we consider a periodic solution~$\mathbf{q}^*(t)$ and solve the second equation in~\eqref{eq:eqm_counter3} to obtain 
\begin{equation}
q_2^*(t)=A(\omega_2,c_2,a,\Omega)\sin(\Omega t-\psi(\omega_2,c_2,a,\Omega)),
\label{eq:c3_q2tilde}
\end{equation}
 where the amplification factor~$A(\omega_2,c_2,a)$ and the phase shift~$\psi(\omega_2,c_2,a)$ are constants depending on the damping coefficient and eigenfrequency, as well as the forcing amplitude and frequency, as indicated. The exact form of $A$ and $\psi$ can be determined from linear theory (see, e.g., G\'{e}radin and Rixen~\cite{Rixen}). Substituting equation~\eqref{eq:c3_q2tilde} into the first equation in~\eqref{eq:eqm_counter3}, we obtain
 \begin{equation}
 \ddot{q}_1^*+c_1\dot{q}_1^*+q_1^*(\omega_1^2+\frac{\kappa A^2}{2}-\frac{\kappa A^2}{2} \cos(2\Omega t -2 \psi))=f_1(t),
 \label{eq:T_vary_stiffness}
 \end{equation}
which is a modification of classic forced-damped Matthieu equation (cf., Guckenheimer and Holmes~\cite{G+H} for the undamped-unforced limit, or Nayfeh and Mook~\cite{Nayfeh_Mook} for the unforced limit). Compared to the standard Matthieu equation, an additional term~$q_1^*\kappa A^2/2$ arises in equation~\eqref{eq:T_vary_stiffness}. For the unforced Matthieu equation ($f_1\!=\!0$), a change of stability of the trivial solution is commonly observed for various values of the damping, stiffness and forcing frequency.  Utilizing this observation, we can prove the nonexistence of a periodic orbit in system~\eqref{eq:eqm_counter3}  with the following fact:

  \begin{fact}
  \label{thm:PO_tvary}
  If  the trivial solution of the system
    \begin{equation}
  \ddot{q}_1^*+c_1\dot{q}_1^*+q_1^*( k_1+\frac{\kappa A^2}{2}-\frac{\kappa A^2}{2} \cos(2\Omega t -2 \psi))=0,
  \label{eq:mat_unforced}
  \end{equation}
  is unstable for some parameter values~$a$,~$\Omega$,~$c_2>0$,~$c_1>0$,~$\omega_1$,~$\omega_2$ and~$\kappa$, then we can find a~$T$-periodic forcing~$f_1(t)$ satisfying condition~\eqref{eq:conv_cond_eval}, such that system~\eqref{eq:T_vary_stiffness} has no periodic orbit.   
  \end{fact}
  \begin{proof}
  The proof relies on the fact that  a~$T$-periodic solution to system~\eqref{eq:T_vary_stiffness} does not exist if a non-trivial~$T$-periodic solution exists in the homogeneous system~\eqref{eq:mat_unforced} and additional  orthogonality conditions between the external forcing and non-trivial $T$-periodic solutions are violated (cf. Farkas~\cite{Farkas_PO}). In Appendix~\ref{app:T_vary}, we show the existence of non-trivial~$T$-periodic solutions~\eqref{eq:mat_unforced} and show that the orthogonality conditions are generally violated for appropriately chosen $f_1$.
  \end{proof}
  We can use the above fact to establish the nonexsistence of a periodic orbit for system~\eqref{eq:eqm_counter3}. To this end, we have to find a set of parameters for which the  trivial solution of  system~\eqref{eq:mat_unforced} is unstable. For the non-dimensional parameters
  \begin{equation}
  \omega_2=1,~~c_1=c_2=0.01,~~\kappa=1,~~\Omega=1,
  \label{eq:pars}
  \end{equation} 
   we calculate the monodromy matrix for the equilibrium at the origin using numerical integration, covering a parameter range for the forcing amplitude~$a$ and the eigenfrequency~$\omega_1$. We depict the result of the Floquet analysis performed on the monodromy matrix in Fig.~\ref{fig:Stab_map}, where we indicate the system configurations with stable trivial solution in green, while red indicates a system configuration with an unstable trivial solution. The critical  system configurations can be found on the stability boundary, which is highlighted in black in Fig.~\ref{fig:Stab_map}. As we prove in Appendix~\ref{app:T_vary}, for the black configurations, we can find a continuous~$T$-periodic forcing~$f_1$ such that the system~\eqref{eq:T_vary_stiffness}, and hence system~\eqref{eq:eqm_counter3}, has no periodic orbit. 
   
   \begin{figure}[ht!]
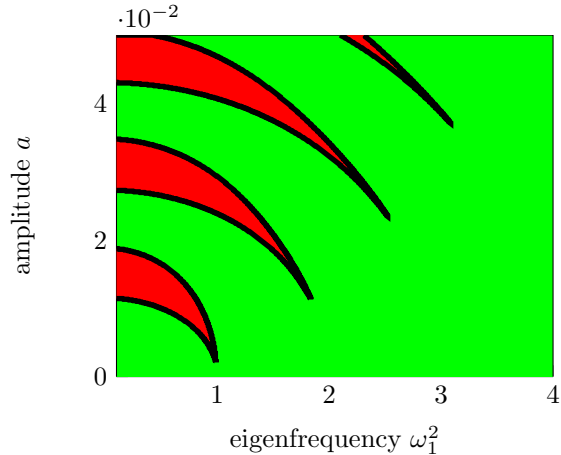

   	\begin{center}
%
   	\end{center}
   	\caption{Stability map of the trivial solution of system~\eqref{eq:mat_unforced} with parameters~\eqref{eq:pars}. Red denotes the instability of the trivial solution, while green denotes a stable trivial solution. Black lines indicate the stability boundary. At the latter parameter values one of the Floquet multipliers is equal to one in norm.}
   	\label{fig:Stab_map}
   \end{figure}
     
\end{example}

\begin{example}[geometric nonlinearities with global extrema]
Condition~\ref{cond:sign_cond} requires the sign of the quantities $q_j(S_j(\mathbf{q})-\bar{f}_j)$ to be constant and non-zero  for $|q_j|\!>\!r$. If the geometric nonlinearities  minus the mean forcing has a constant sign outside the region $|q_j|\!>\!r$ for some degree of freedom ($\sign(S_j(\mathbf{q}) - \bar{\mathbf{f}}_j)=const.$ for $|q_j|\!>\!r$), then the quantities  $q_j(S_j(\mathbf{q})-\bar{f}_j)$  evaluated for $q_j\!>\!r$  and for $q_j\!<\!-r$ have opposite sign. Therefore, condition~\ref{cond:sign_cond} is violated.  This is certainly the case if the geometric nonlinearities have a global minimum  value and the mean forcing of a single coordinate is below that minimum value, i.e.,
\begin{equation}
S_j(\mathbf{q})>S_{\min}>\bar{f}_l, \quad j=1,...,N, \quad 1\leq l\leq N, \qquad \mathbf{q}\in\mathbb{R}^N.
\label{eq:global_min}
\end{equation} 
Then $\left(S_l(\mathbf{q}) - \bar{\mathbf{f}}_l\right)$ is always positive and  $q_l\left(S_l(\mathbf{q}) - \bar{\mathbf{f}}_l\right)$ changes sign. For system~\eqref{eq:sys0} with geometric nonlinearities and mean forcing satisfying~\eqref{eq:global_min}, we have the following fact:
\begin{fact}
	\label{thm:global_min}
	If the geometric nonlinearities and the mean forcing of system~\eqref{eq:sys0} satisfy the conditions~\eqref{eq:global_min}, then no periodic orbit exists for system~\eqref{eq:sys0}. 
\end{fact}
\begin{proof}
	We detail this proof in Appendix~\ref{app:global_extrm}.
\end{proof}
 \begin{remark}
	The conclusion of Fact~\eqref{thm:global_min} also holds in the case of geometric nonlinearities with global maxima  $S_{max}$ and a mean forcing larger than  $S_{max}$, i.e. for systems satisfying, 
	\begin{equation}
	S_j(\mathbf{q})<S_{\max}<\bar{f}_l, \quad j=1,...,N, \quad 1\leq l\leq N \qquad \mathbf{q}\in\mathbb{R}^N.
	\label{eq:global_max}
	\end{equation} 
	hold.
\end{remark}

An example for a nonlinearity satisfying~\eqref{eq:global_min} and~\eqref{eq:global_max}  is the simple pendulum,  whose   geometric nonlinearity~$S(q)\!=\!c_p\sin(q)$  has global maximum  and minimum value. Therefore, the damped forced-pendulum 
	\begin{equation}
	\ddot{q}+c\dot{q}+c_p\sin(q)=f(t),\qquad |\bar{f}|>|c_p|,
	\label{eq:pendulum}
	\end{equation}
	has no $T$-periodic solution. 
\end{example}

The previous example indicates that the mean value of the forcing plays a critical role in the existence of periodic orbits. One might wonder if a zero-mean restriction of the forcing, as in the theorem of Rouche and Mawhin~\cite{Rouche_PO} (cf. Theorem~\ref{thm:RM} in Appendix~\ref{app:Exist_proof}), allows relaxing some of our conditions, notably condition~\ref{cond:sign_cond}. In the following example, we show that even for zero-mean forcing ($\bar{f}_j=0$) condition~\ref{cond:sign_cond} cannot be relaxed.

\begin{example}[Constant-sign geometric nonlinearities, zero-mean forcing]
We consider the nonlinear oscillator
\begin{equation}
\ddot{q}+c\dot{q}+\omega^2q+\kappa q^2=f \cos(\Omega t),
\label{eq:eqm_x^2}
\end{equation}
the simplest example with  geometric nonlinearities  violating condition~\ref{cond:sign_cond}. As forcing, we choose simple single harmonic forcing with amplitude $f$. For system~\eqref{eq:eqm_x^2}, we have then the following fact:
\begin{fact}
	\label{thm:xsprt_NL}
If the forcing amplitude $f$ for mechanical system~\eqref{eq:eqm_x^2} is above the threshold 
\begin{equation}
|f|>\frac{ \omega^2}{ 2|\kappa|} \left(|-\Omega^2+\mathrm{i}c\Omega+\omega^2|+ 2\omega^2 \right)+|\kappa|\frac{\omega^4}{4\kappa^2},
\label{eq:c4_f_thres}
\end{equation}
then no $T$-periodic solution to system~\eqref{eq:eqm_x^2} exist.
\end{fact}
\begin{proof}
We detail the proof in Appendix~\ref{app:xsprt_NL} . 
\end{proof}
Therefore, choosing any forcing amplitude exceeding the threshold~\eqref{eq:c4_f_thres} will necessarily rule out the existence of a periodic orbit. 
\end{example}

\subsection{Examples with periodic orbits guaranteed by Theorem~\ref{thm:Existence}.}
In the following, we give examples in which Theorem~\ref{thm:Existence} guarantees the existence of a  periodic response. Since the damping condition~\ref{cond:damping} and the assumption ~\ref{cond:potential} on the geometric nonlinearities  derived from a potential are simple to verify, we focus on condition~\ref{cond:sign_cond} and \ref{cond:conv_cond}. We start with the classic Duffing oscillator and proceed with higher-dimensional examples.

\begin{example}[Duffing oscillator]
	The forced-damped Duffing oscillator is simple harmonic oscillator with an additional cubic nonlinearity added, i.e.
	\begin{equation}
	\ddot{q}+c\dot{q}+\omega^2q +\kappa q^3=f \cos(\Omega t), 
	\label{eq:duffing}
	\end{equation}
	where we have chosen single harmonic forcing with amplitude~$f$ and frequency~$\Omega$.

	For~$\kappa\!\geq \!0$, the potential of eq.~\eqref{eq:duffing} is positive definite for all~$q$ (cf. Fig.~\ref{fig:Duf_pot_pos}). Therefore, condition~\ref{cond:conv_cond} is trivially satisfied. Furthermore, condition~\eqref{eq:sign_cond} is satisfied for arbitrarily small radius~$r$, which can be used for the  upper bound on the amplitudes~\eqref{eq:bnd_ampls}. Therefore, both Theorems~\ref{thm:Existence} and~\ref{thm:posdef} apply and guarantee the existence of a  periodic solution without any numerics. We compute  periodic responses with the automated continuation package \textsc{coco}~\cite{coco} and show the amplitudes in Fig.~\ref{fig:Duf_FRF_pos}.

		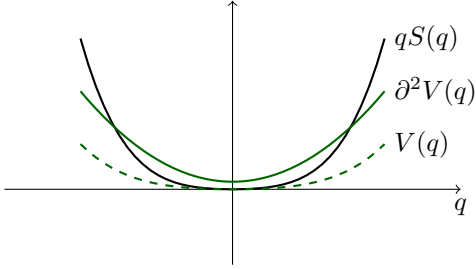
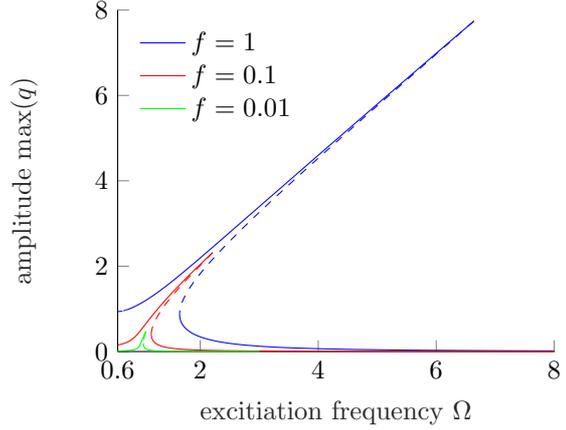
\begin{figure}[ht!]
		\begin{center}
			\begin{subfigure}[b]{0.49\textwidth}
		\begin{tikzpicture}
		\draw[->] (-1,2) -- (5,2) node[below]{$q$};
		\draw[->] (2,1) -- (2,4.5);
		
		\draw [ thick,  domain=-2:2, samples=40] 
		plot ({2+\x}, {2+\x*(0.1*\x+0.1*\x*\x*\x)} )node[right]{$qS(q)$};
		
		\draw [ thick, black!60!green,   domain=-2:2, samples=40] 
		plot ({2+\x}, {2+0.1+0.1*3*\x*\x} ) node[right]{\textcolor{black}{$\partial^2V(q)$}};

		\draw [ thick,dashed, black!60!green,   domain=-2:2, samples=40] 
		plot ({2+\x}, {2+(0.1/2*\x*\x+0.1/4*\x*\x*\x*\x)} ) node[right]{\textcolor{black}{$V(q)$}};
		
		\end{tikzpicture}
		
				\caption{Potential and nonlinearity for~$\kappa\geq 0$ of system~\eqref{eq:duffing}.
				\newline
				\newline}
				\label{fig:Duf_pot_pos}
			\end{subfigure}
			\begin{subfigure}[b]{0.49\textwidth}
			\begin{tikzpicture}
			
			\begin{axis}[%
			width=4.521in,
			height=3.566in,
			at={(0.758in,0.481in)},
			scale only axis,
			scale=0.5,
			unbounded coords=jump,
			xmin=0.6,
			xmax=8,
			xtick={0.6,   2,   4,   6,   8},
			xlabel style={font=\color{white!15!black}},
			xlabel={excitiation frequency $\Omega$},
			ymin=0,
			ymax=8,
			ytick={0, 2, 4, 6, 8},
			ylabel style={font=\color{white!15!black}},
			ylabel={amplitude $\max(q)$},
			axis background/.style={fill=white},
			axis x line*=bottom,
			axis y line*=left,
			legend style={at={(0.03,0.97)}, anchor=north west, legend cell align=left, align=left, draw=none, fill=none}
			]
			\addplot [color=blue, dashed, forget plot]
			table[row sep=crcr]{%
				0.690523210767485	0.952337939006846\\
				nan	nan\\
				0.698182301926214	0.954888303880584\\
				0.740720652509768	0.971631990923166\\
				nan	nan\\
				6.63849714585131	7.74223765222456\\
				6.62171798158195	7.72090024280358\\
				6.44651182537539	7.50828464308119\\
				4.91577896128427	5.65561585331117\\
				4.33599414591443	4.94578887901755\\
				3.87510390090033	4.37509096413009\\
				3.50113111194255	3.90554318332616\\
				3.19265400004492	3.51168732418044\\
				2.93728247058347	3.17907828191385\\
				2.71964864498644	2.88884989514934\\
				2.53833310087801	2.64026105818872\\
				2.38436606447296	2.42229227951784\\
				2.25128045651764	2.22669040588663\\
				2.13655502675336	2.05049543262114\\
				2.04267066338799	1.89884240143981\\
				1.95963459076642	1.75663278873508\\
				1.88865386402217	1.62624563596382\\
				1.83059907182369	1.51048751121727\\
				1.78576928475954	1.41247561578501\\
				1.7485929154143	1.32222352681433\\
				1.71907302344068	1.24113457460605\\
				1.69272060269566	1.1562438759459\\
				1.67415087393098	1.08233914883762\\
				1.6622832616539	1.02099954219322\\
				1.65386733625921	0.957589620998712\\
				1.64986912998379	0.892037560921658\\
				1.64969964761111	0.875978383463963\\
			};
			\addplot [color=blue]
			table[row sep=crcr]{%
				0.6	0.9384967793951\\
				0.618338628224642	0.938104950308617\\
				0.648793905798994	0.941590613821834\\
				0.681348297080721	0.94949717529313\\
				nan	nan\\
				0.690523210767484	0.952337939006846\\
				nan	nan\\
				0.740720652509767	0.971631990923166\\
				0.784735397167932	0.992788129895645\\
				0.836430015672978	1.02166816622835\\
				0.903081285457935	1.06429766501611\\
				0.966336533788471	1.10960831859865\\
				1.0404471944595	1.16795720933822\\
				1.12272640496747	1.23856764031928\\
				1.21096166306183	1.32017669812361\\
				1.31459594917042	1.42249797672993\\
				1.42973088024528	1.54271419364308\\
				1.56258598538301	1.68788906850593\\
				1.72506712218724	1.87204566881591\\
				1.93824470663112	2.1206274135131\\
				2.24332234561217	2.48367344652836\\
				2.82191517709147	3.18020841398973\\
				4.09040582555806	4.70711218860191\\
				5.1301513553287	5.95135853456293\\
				6.20470206456904	7.23075673888764\\
				6.63839898395265	7.74221443674985\\
				6.63849719810129	7.74223885939222\\
				nan	nan\\
				1.64969964761786	0.875978168682112\\
				1.65155148549245	0.82430636665417\\
				1.65752493035228	0.772085782515534\\
				1.66852177696571	0.718685351474399\\
				1.68567340824022	0.664176145770224\\
				1.70122143200804	0.627270357931334\\
				1.72065870423761	0.589962383306069\\
				1.74461095101345	0.552308634863023\\
				1.77383418367524	0.514378603671281\\
				1.80925287407706	0.476256600581582\\
				1.85201297566793	0.438043809801835\\
				1.90355649004551	0.399860855267345\\
				1.96572746263299	0.361851205319965\\
				2.04092385223854	0.324185924125231\\
				2.1323157324608	0.287070543246399\\
				2.24415638575965	0.250755191413058\\
				2.30947358146801	0.232991434450959\\
				2.38221244762552	0.215549828627013\\
				2.46342030073956	0.198481934977741\\
				2.55430938007594	0.181845555353805\\
				2.65626865728771	0.165706640537511\\
				2.77086264462912	0.150139592524356\\
				2.89980578779873	0.13522752505212\\
				3.04489677404538	0.121061442991419\\
				3.20789563100687	0.107737495024908\\
				3.39033368691904	0.0953513278079097\\
				3.59326939279762	0.0839889438486576\\
				4.06111087945702	0.0645594290095293\\
				4.60381920839382	0.0495209381140764\\
				5.20392010977797	0.0383437002364495\\
				5.8426944211996	0.0301780645002339\\
				6.84339492707523	0.021818841981327\\
				8.21746165455097	0.0150315492033481\\
			};
			\addlegendentry{$f=1$}
			
			\addplot [color=red, dashed, forget plot]
			table[row sep=crcr]{%
				2.21142996842709	2.31700508517089\\
				2.19526116357129	2.29320902127495\\
				1.77806966705553	1.70279478893068\\
				1.65768537817321	1.5207165061733\\
				1.56130738425346	1.36771287090931\\
				1.47735651684721	1.22682619365168\\
				1.40778852042788	1.10216064136085\\
				1.34632324943426	0.982966017255282\\
				1.30144016298455	0.887650452163636\\
				1.25886747834914	0.786685570868225\\
				1.22649530277958	0.697889961985493\\
				1.20382182958141	0.623657468396525\\
				1.18541656485231	0.546538668215415\\
				1.17557605016229	0.486830334584341\\
				1.17049544254077	0.425621207293168\\
				1.1701620353421	0.406571951072478\\
			};
			\addplot [color=red]
			table[row sep=crcr]{%
				0.6	0.152489831626079\\
				0.647927694479614	0.166772142044398\\
				0.693806195445763	0.184229573757699\\
				0.740193715011914	0.207012313358405\\
				0.779522382263467	0.231763717754115\\
				0.816763635792036	0.261166403525133\\
				0.854319984397806	0.298017514752178\\
				0.887033862897949	0.336768317016729\\
				0.92924460291885	0.395938504698883\\
				0.97780431414156	0.474480688527281\\
				1.05295566104985	0.607403139046054\\
				1.1808213269673	0.834298678908333\\
				1.27425135453525	0.991534528291627\\
				1.37283779904428	1.14950632779183\\
				1.48377801391561	1.31941373951405\\
				1.61241883236415	1.50838326562746\\
				1.75482412288165	1.70997799638529\\
				1.92679021768477	1.94546170057186\\
				2.10409014191759	2.18085570163233\\
				2.21027850107486	2.31642148425437\\
				2.21154382561617	2.31737642925583\\
				nan	nan\\
				1.17016470713563	0.404895993506809\\
				1.17210312987449	0.363035218065562\\
				1.17839181945543	0.320672975452563\\
				1.19055815115329	0.277952942361956\\
				1.1995508580297	0.256516682831522\\
				1.21102492590957	0.235079104065887\\
				1.22554433315896	0.213696287345121\\
				1.24385621361276	0.192438770776089\\
				1.26695770479943	0.171411668114688\\
				1.29617979032515	0.1507654949572\\
				1.33327751121934	0.130707451850315\\
				1.38049331847831	0.111521468800611\\
				1.43085864742045	0.0960899707665615\\
				1.48443141925101	0.0834319562293668\\
				1.54293180043084	0.0726247423111079\\
				1.60836492780066	0.0631271171774639\\
				1.68303919272256	0.0546256134587519\\
				1.87082038248595	0.0400157335508489\\
				2.12546445211435	0.0284314087162851\\
				2.44979049735887	0.019994226244556\\
				2.8394349729313	0.0141593122425245\\
				3.60663937152344	0.00832777234897009\\
				4.93476335413663	0.00428227050993613\\
				7.72173868977041	0.00170573766097704\\
				8.07345351071217	0.00155809092880332\\
			};
			\addlegendentry{$f=0.1$}
			\addplot [color=green, dashed, forget plot]
			table[row sep=crcr]{%
				1.078341813099	0.465211614978717\\
				1.07492200563165	0.446616666096186\\
				1.0508527097915	0.332397817552563\\
				1.04158988868221	0.268431730659876\\
				1.0374626798853	0.215101890070264\\
				1.03703758379492	0.193877711726034\\
			};
			\addplot [color=green]
			table[row sep=crcr]{%
				0.6	0.0156173964103936\\
				0.730362758194566	0.0214073394564123\\
				0.795285615885699	0.0271435850806752\\
				0.832016173290073	0.0323608462728791\\
				0.871622564702805	0.0412910020404631\\
				0.909048316978487	0.0565129091104071\\
				0.931006895178305	0.0722665515287146\\
				0.947745972743621	0.0911964076306551\\
				0.958614706389395	0.108794671943433\\
				0.970044311313822	0.133526936765882\\
				0.987487233956283	0.184749592063437\\
				1.05852065537995	0.419928419161574\\
				1.07719426347043	0.465769906459806\\
				1.07834181310019	0.465211738277308\\
				nan	nan\\
				1.03703758157269	0.193861317081506\\
				1.0383743282292	0.1597611878239\\
				1.04369877369177	0.124283061700385\\
				1.05376594706206	0.0944801647316029\\
				1.06654255201032	0.0740206136599757\\
				1.07999891887576	0.0605848363745674\\
				1.09463655624476	0.0506293983138852\\
				1.11199479099772	0.0423217415039838\\
				1.13468496713133	0.0347782085690187\\
				1.16808935675762	0.0274263297138679\\
				1.22617206088552	0.0198485686074799\\
				1.27623059663951	0.0158937329391242\\
				1.35699051589863	0.0118762184436263\\
				1.58555501388757	0.00660264359273821\\
				2.07713987571579	0.00301665040320787\\
				3	0.00124994850400695\\
			};
			\addlegendentry{$f=0.01$}
			\end{axis}
			\end{tikzpicture}%
				\caption{Frequency-response curves for the Duffing oscillator~\eqref{eq:duffing}; solid lines mark stable and dashed lines unstable  periodic orbits; parameters:~$c\!=\!0.01$,~$\omega^2\!=\!1$ and~$\kappa\!=\!1$.}
				\label{fig:Duf_FRF_pos}
			\end{subfigure}
		\end{center}
		\caption{Features of the Duffing oscillator~\eqref{eq:duffing} for $\kappa\geq 0$ (hardening spring stiffness).}
		\label{fig:Duffing_pos}
	\end{figure}

	For negative values of the coefficient~$\kappa$, the Hessian of the potential is not globally positive (cf. Fig~\ref{fig:Duf_pot_neg} green curve). However, outside the ball of radius~$r^*\!=\! \omega\sqrt{-1/(3\kappa)}$, the second derivative of the potential is negative. Therefore, the existence of a periodic orbit is guaranteed by Theorem~\ref{thm:posdef}. Furthermore, two nontrivial equilibria arise at~$q_0=\omega\sqrt{-1/\kappa}$ (cf. Fig~\ref{fig:Duf_pot_neg}, black curve). If we select this $q_0$ as the radius $r$ in Theorem~\ref{thm:Existence}, then condition~\ref{cond:sign_cond} is satisfied. Again, both Theorems~\ref{thm:Existence} and~\ref{thm:posdef} guarantee the existence of a periodic response. 
	
	We numerically continue the trivial and the two nontrivial periodic orbits for increasing forcing amplitude at the fixed forcing frequency~$\Omega\!=\!1$ and show the amplitudes of the  periodic response in Fig.~\ref{fig:Duf_bifi_neg}. For larger forcing amplitudes, numerical continuation becomes more challenging, yet our results continue to imply the existence of a periodic response rigorously. 
	
		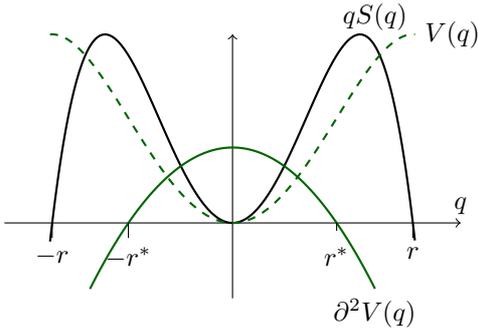
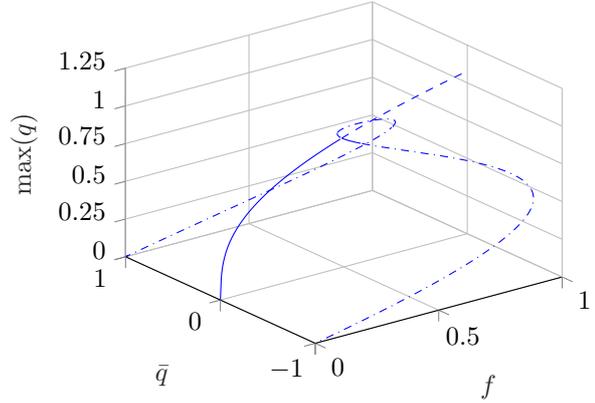
\begin{figure}[ht!]
		\begin{center}
			\begin{subfigure}[b]{0.49\textwidth}
				\begin{tikzpicture}
				\draw[->] (-1,2) -- (5,2) node[above]{$q$};
				\draw[->] (2,1) -- (2,4.5);
				
				\draw [ thick,  domain=-3.2:3.2, samples=200] 
				plot ({2+0.75*\x}, {2+\x*(1*\x-0.1*\x^3)} );
				\draw (2+0.75*2.5,2+2.2361+0.2) node[above]{$qS(q)$};
				
				\draw [ thick, black!60!green, domain=-2.5:2.5, samples=200] 
				plot ({2+0.75*\x}, {2+1-0.1*3*\x*\x} ) node[below]{\textcolor{black}{$\partial^2V(q)$}};
				
				\draw [ thick, dashed ,black!60!green, domain=-3.2:3.2, samples=200] 
				plot ({2+0.75*\x}, {2+(1/2*\x*\x-0.1/4*\x*\x*\x*\x)} )node[right]{\textcolor{black}{$V(q)$}};
				
				\draw  (2+0.75*3.1623,2) -- (2+0.75*3.1623,1.8)node[below]{$r$};
				\draw  (2-0.75*3.1623,2) -- (2-0.75*3.1623,1.8)node[below]{$-r$};
				
				\draw[dashed]  (2+0.75*1.8257,2) -- (2+0.75*1.8257,1.8) node[below]{$r^*$};
				\draw  (2-0.75*1.8257,2) -- (2-0.75*1.8257,1.8)node[below]{$-r^*$};
				
				\end{tikzpicture}
				\caption{Potential and nonlinearity for~$\kappa<0$	of system~\eqref{eq:duffing}.	\newline
					\newline \newline \newline}
				\label{fig:Duf_pot_neg}
			\end{subfigure}
			\begin{subfigure}[b]{0.49\textwidth}
\begin{tikzpicture}

\begin{axis}[%
width=4.521in,
height=3.566in,
at={(0.758in,0.481in)},
scale only axis,
scale=0.5,
unbounded coords=jump,
xmin=0,
xmax=1,
xtick={   0,    0.5,      1},
tick align=outside,
xlabel style={font=\color{white!15!black}},
xlabel={ $f$},
ymin=-1,
ymax=1,
ylabel style={font=\color{white!15!black}},
ylabel={ $\bar{q}$},
zmin=0,
zmax=1.25,
ztick={   0, 0.25,  0.5,   0.75,   1,1.25 },
zlabel style={font=\color{white!15!black}},
zlabel={ $\max(q)$},
view={-37.5}{30},
axis background/.style={fill=white},
axis x line*=bottom,
axis y line*=left,
axis z line*=left,
xmajorgrids,
ymajorgrids,
zmajorgrids,
legend style={at={(0.03,0.97)}, anchor=north west, legend cell align=left, align=left, draw=none}
]
\addplot3 [color=blue, dashed, forget plot]
table[row sep=crcr] {%
	0.491811132907519	-2.79594529883909e-06	0.856583959102101\\
	0.521766619980119	-4.99452989033422e-06	0.873216309546953\\
	0.559780363317188	-7.52507978063477e-06	0.893400753405359\\
	0.599485725928594	-9.90405867873356e-06	0.913497801986436\\
	0.64090626012005	-1.2142333071008e-05	0.933506504246018\\
	0.684064577854211	-1.42498322346363e-05	0.953426014402961\\
	0.728982356868731	-1.62356435946531e-05	0.973255585660433\\
	0.77568034794176	-1.49714940269607e-05	0.992997701019363\\
	0.82417838318861	-1.28640470746788e-05	1.01264939566801\\
	0.874495385276433	-1.08831646501439e-05	1.03220922373206\\
	0.926649377446136	-9.01979712875089e-06	1.05167679566876\\
	1	-6.67606485760253e-06	1.07781063673299\\
};
\addplot3 [color=blue ]
table[row sep=crcr] {%
	1.59859228077153e-08	-2.64066546407093e-12	7.99286149932854e-07\\
	0.003140969151599	3.24476879032254e-07	0.131606330814786\\
	0.0050067335860744	5.17070998240499e-07	0.169771683540038\\
	0.00727031907854581	-1.80588250253688e-06	0.200506269912602\\
	0.0107999637813091	6.87950869160403e-07	0.234648641579765\\
	0.0159785005664769	-8.69115041601809e-07	0.271181155079281\\
	0.0233019108319426	3.0496530147861e-06	0.309909830488807\\
	0.0279150007620785	-5.71982377906011e-06	0.329928759883292\\
	0.0332349018475669	1.73963865057747e-06	0.350289153733721\\
	0.0393125270254444	4.10078895818344e-06	0.370916938993056\\
	0.0461964913723374	-1.20768809841465e-07	0.391753320699345\\
	0.0539335544861506	-6.83068016427146e-06	0.412744637832691\\
	0.0625689882926795	2.52405512224385e-06	0.433853685168149\\
	0.0721468510354776	-1.32858606005648e-06	0.455033344171835\\
	0.0827101753715394	1.55835248316638e-06	0.476271172936948\\
	0.0943010882433336	-6.67347706473365e-06	0.497526392516645\\
	0.106960880861968	9.52413304222688e-06	0.518794978770087\\
	0.120730044188975	-1.354942360976e-06	0.540064044574185\\
	0.135648281582631	-1.03152310857446e-05	0.561299790688363\\
	0.15175450695455	2.96307689462427e-06	0.582516310222075\\
	0.169086834197124	9.64800540748101e-06	0.603690717079505\\
	0.18768256176488	2.25362540240415e-06	0.624825763400573\\
	0.207578154984529	-4.65478398647967e-06	0.645902664994893\\
	0.228809227776355	-1.11184131305064e-05	0.666917686968866\\
	0.25141052487349	-7.28115537351748e-06	0.687877366230762\\
	0.275415905225984	1.30276449739064e-07	0.70877198322213\\
	0.300858327012236	7.06545034301342e-06	0.729594781358513\\
	0.327769834504321	1.35626521852172e-05	0.750343346728149\\
	0.356181546915055	1.03926480525818e-05	0.771024752112304\\
	0.386123649276038	6.89999317804446e-06	0.791627693953253\\
	0.417625385340255	3.63164376315783e-06	0.812149845543658\\
	0.450715052466524	5.70098260821439e-07	0.83258960851068\\
	0.485419998417005	-2.30049719129877e-06	0.852945533598099\\
};
\addplot3 [color=blue, dashdotted, forget plot]
table[row sep=crcr] {%
	0.01	0.999988924135885	0.00333325640684579\\
	6.38810634390197e-07	1.00000000000433	2.12931738019506e-07\\
	0.0954680574451847	0.998984421352075	0.0318714253548349\\
	0.224318898195838	0.994307296508797	0.0754318012208557\\
	0.4639438639814	0.973999316908629	0.161027934876151\\
	0.619590475180725	0.949820370283037	0.223411474428582\\
	0.708966901765402	0.929812705538005	0.263936685663652\\
	0.785700036785072	0.906959159034504	0.30350412932197\\
	0.849301510325445	0.881502350977657	0.342035168637233\\
	0.899684125730519	0.853693145949996	0.379462440536023\\
	0.919990541546127	0.838985862070278	0.397744041146386\\
	0.93711918094172	0.823785906294425	0.415729179050919\\
	0.951151265628233	0.808124962897557	0.433412205816974\\
	0.962183506732865	0.792034487941091	0.450787946574589\\
	0.970326276558513	0.775545641231456	0.467851633691349\\
	0.975701819354535	0.758689235448335	0.484598836509911\\
	0.978442526294685	0.741495701238817	0.501025389174248\\
	nan	nan	nan\\
	0.978875254951206	0.7298737933461	0.511784527671766\\
	0.978689293329639	0.723995066360037	0.517127318671406\\
	0.976589974169951	0.706216946319051	0.532900775245281\\
	0.972297934684706	0.688190543425351	0.548341967306407\\
	0.965970709729817	0.669944650754051	0.563447102873913\\
	0.95776875903846	0.651507657228494	0.578212339448283\\
	0.947854315444304	0.632907549887888	0.59263374402286\\
	0.936390316438116	0.614171909418831	0.606707264699387\\
	0.923539408902872	0.5953278952178	0.62042871507055\\
	0.909463016807394	0.576402216628565	0.633793772166313\\
	0.89432046259338	0.557421087573306	0.646797988328751\\
	0.878268134847497	0.538410162572529	0.659436816883503\\
	0.861458697438017	0.519394453111099	0.671705650927596\\
	0.826156060169146	0.481444874062079	0.695114920575447\\
	0.789531901213056	0.443757282987595	0.716991950343494\\
	0.752602749587456	0.406495536460602	0.737303425178585\\
	0.716266248791765	0.369799542053087	0.756020917151916\\
	0.681290600578266	0.333782176995973	0.773124408906027\\
	0.648310615082782	0.298524642305122	0.788602544524318\\
	0.617830745624808	0.264074448818741	0.802453035375958\\
	0.590234544661214	0.230445233406076	0.814682138599702\\
	0.565799136562581	0.197618312944471	0.825303337530747\\
	0.544712808891385	0.165540458719235	0.834340615652713\\
	0.535464716268766	0.149762623861588	0.838269771236659\\
	0.524215836880591	0.128394490635486	0.843023234266873\\
	nan	nan	nan\\
	0.491811128030878	8.69551628102094e-06	0.856576377325514\\
	0.491811128030946	8.51850386629494e-06	0.85657637732137\\
	nan	nan	nan\\
	0.524215844290894	-0.128372290849479	0.843025664304834\\
	0.537429392647025	-0.153205908036765	0.83744343470746\\
	0.557255945410639	-0.185127782948417	0.82899135769348\\
	0.59352123191281	-0.234591413972708	0.813254861785224\\
	0.742269481386199	-0.396096069385327	0.742748839552132\\
	0.780405568479133	-0.434518740461415	0.7221609890378\\
	0.818646751969129	-0.473625466236908	0.699775408297069\\
	0.837458278065757	-0.493378222445108	0.687914281609779\\
	0.855857230776105	-0.513229172379046	0.675612023347358\\
	0.873677019545743	-0.533150001388891	0.662872608441077\\
	0.890742749356515	-0.553110858446227	0.649700409167101\\
	0.906872218155041	-0.573080566022035	0.636100094531536\\
	0.921877002445145	-0.593026804033114	0.62207654222254\\
	0.9355636237975	-0.61291626224561	0.607634769184594\\
	0.947734796768078	-0.632714758573435	0.592779885604729\\
	0.958190766538928	-0.652387323742863	0.57751707565189\\
	0.966730751560689	-0.671898255587829	0.5618516068002\\
	0.973154511747277	-0.691211148642219	0.545788868080898\\
	0.977264065633751	-0.710288906611804	0.529334436202078\\
	0.978865579789146	-0.729091711217784	0.512496202652743\\
	nan	nan	nan\\
	0.978880533932134	-0.731049964332148	0.510704820104363\\
	0.977771450249824	-0.747582912504278	0.495278602960103\\
	0.973802588538238	-0.765723875520674	0.477687918661861\\
	0.96679091388334	-0.783474850282176	0.45973158933248\\
	0.956582038731277	-0.800795531564395	0.441417870683271\\
	0.943038117020024	-0.817645167890532	0.422755979880346\\
	0.926040804877259	-0.833982714883477	0.403756230993155\\
	0.905494262717616	-0.849767036298674	0.384430152692267\\
	0.881328107939289	-0.864957153082543	0.364790580569822\\
	0.853500210672962	-0.879512537486304	0.344851717212285\\
	0.821999213591402	-0.893393445818231	0.324629154464716\\
	0.78684665277393	-0.906561280119579	0.304139854185256\\
	0.748098561641683	-0.918978966231103	0.28340208611395\\
	0.705846454762794	-0.930611333708103	0.262435324111388\\
	0.660217612438232	-0.941425482121767	0.241260104729805\\
	0.611374618695657	-0.951391118632675	0.219897854576496\\
	0.559514141836031	-0.96048085335273	0.19837069495463\\
	0.504864984377529	-0.968670441788578	0.176701233571008\\
	0.447685464351951	-0.975938967275117	0.154912353549212\\
	0.38826021903057	-0.982268960354003	0.133027009540732\\
	0.326896542832468	-0.987646456073327	0.111068039469548\\
	0.263920382230263	-0.992060993752812	0.0890579985672058\\
	0.199672112018974	-0.995505566544864	0.0670190200843606\\
	0.134502210754941	-0.997976529916363	0.0449727046996078\\
	0.00282411978009467	-0.999999104018307	0.000941336625593925\\
	0	-1	1.33226762955019e-15\\
};
\addplot3 [color=blue, forget plot]
table[row sep=crcr] {%
	0.978880529731251	0.731036000252652	0.510720498353465\\
	0.978880529731258	0.73103593567001	0.510720557552641\\
	nan	nan	nan\\
	0.51960673878924	0.118672973212244	0.844963219570841\\
	0.513008282516832	0.103335064737464	0.847733168570446\\
	0.507301266036917	0.0881160312285469	0.85012218506845\\
	0.502487245728766	0.0730004939334814	0.852132727374487\\
	0.498566823537944	0.0579723467420531	0.853767003479119\\
	0.495539968425999	0.0430148594961866	0.855026920624501\\
	0.493406290233415	0.0281107746077807	0.855914039206657\\
	0.492165265693013	0.0132423977991154	0.856429531107927\\
	nan	nan	nan\\
	0.491816416362257	-0.00160831586228116	0.856574142594225\\
	0.493794280980634	-0.0313301913887234	0.855751392398114\\
	0.499340628377575	-0.0612031001311801	0.853442148279052\\
	0.508459941612617	-0.091375699703358	0.849634575669059\\
	0.521156980966466	-0.12199818818914	0.844312601180367\\
	nan	nan	nan\\
	0.978875259519486	-0.729887877603915	0.511768767403757\\
	0.978880533932252	-0.731049794740056	0.510704975563689\\
};
\end{axis}
\end{tikzpicture}%
				\caption{Amplitude and the mean position $\bar{q}$ of the periodic orbit of system~\eqref{eq:duffing}, depending on the forcing amplitude; solid lines mark stable and dashed lines unstable  periodic orbits; parameters~$c\!=\!0.01$,~$\omega^2\!=\!1$,~$\kappa\!=\!-1$ and~$\Omega\!=\!1$. }
				\label{fig:Duf_bifi_neg}
			\end{subfigure}
		\end{center}
		\caption{Features of the Duffing oscillator~\eqref{eq:duffing} for $\kappa< 0$ (softening spring stiffness).}
		\label{fig:Duffing_neg}
	\end{figure}

\end{example}

\begin{example}[Oscillator chain]
	\label{ex:osci}
	We consider the~$N$-dimensional oscillator chain depicted in Fig.~\ref{fig:Osci_chain}. Two adjacent masses are connected via nonlinear springs and linear dampers. The first and last mass are suspended to the wall. The equation of motion of the~$j$-th mass is given by
	\begin{equation}
	m_j\ddot{q}_j-c_j(\dot{q}_{j-1}-\dot{q}_{j})+c_{j+1}(\dot{q}_{j}-\dot{q}_{j+1})-S_{j}(q_{j-1}-q_{j})+S_{j+1}(q_{j}-q_{j+1})=f_j(t), \qquad j=1,...,N,
	\label{eq:eqm_chain}
	\end{equation}
	where we set the coordinates~$q_0$ and~$q_{N+1}$ to zero. Variants of such systems have been investigated by Shaw and Pierre~\cite{SHAW+Pierre}, Breunung and Haller~\cite{TB_Backbone} and Jain et al.~\cite{NPO}.  Physically, the system may represent, e.g., a discretized beam. As we detail in Appendix~\ref{app:chain_sys}, the following fact holds for the chain system~\eqref{eq:eqm_chain}:
	
	\begin{fact}
		\label{thm:chain_sys}
	For positive masses, damping coefficients and hardening spring stiffnesses, i.e.  
	\begin{equation}
	m_j>0,\qquad c_j>0, \qquad  \frac{\partial S_{j}(\delta)}{\partial \delta}>0 , \qquad j=1,...,N+1,
	\label{eq:pars_chain}
	\end{equation}  
	system~\eqref{eq:eqm_chain} satisfies the conditions of Theorem~\ref{thm:Existence} and hence must have a steady-sate response.	
	\end{fact}
	\begin{figure}[ht!]
		\begin{center}
	\begin{tikzpicture}

\tikzstyle{spring}=[thick,decorate,decoration={zigzag,pre length=0.3cm,post
	length=0.3cm,segment length=6}]
\tikzstyle{damper}=[thick,decoration={markings,  
	mark connection node=dmp,
	mark=at position 0.45 with 
	{
		\node (dmp) [thick,inner sep=0pt,transform shape,rotate=-90,minimum
		width=8pt,minimum height=3pt,draw=none] {};
		\draw [thick] ($(dmp.north east)+(5pt,0)$) -- (dmp.south east) -- (dmp.south
		west) -- ($(dmp.north west)+(5pt,0)$);
		\draw [thick] ($(dmp.north)+(0,-3pt)$) -- ($(dmp.north)+(0,3pt)$);
	}
}, decorate]


\node[pattern=north east lines, pattern color=black] at (-2.375,-0.125) (LW) [minimum width=0.25cm,  minimum height=1.75cm] {};

\node[draw,outer sep=0pt,thick] (M1) [minimum width=1.5cm, minimum height=1cm] {} node[above]{$\qquad m$};

\draw[spring] ($(LW.east) - (0,0.125)$) -- ($(M1.west) - (0,0.25)$) node [midway,below] {$S_1$};
\draw[damper] ($(LW.east) + (0,0.375)$) -- ($(M1.west) + (0,0.25)$)
node [midway,above] {\raisebox{0.1cm}{$c_1$}};
\draw[thick,->] (-1.95,-0.4)-- (-1.15,-0.05) ;

\draw[thick] ($(M1.south) + (-0.5,-0.125)$) circle (0.125);
\draw[thick] ($(M1.south) + (0.5,-0.125)$) circle (0.125);
\draw[-triangle 45] ($(M1) - (0.5,0)$) -- ($(M1) + (0.5,0)$)  node[midway,below]{$f_1$};
\draw[-latex] ($(M1.north) + (-0.5,0.1)$) -- ($(M1.north) + (0.5,0.1)$) node[midway,above]{$q_1$};		

\coordinate (LBs) at (2.25,-0.25);	
\coordinate (LBd) at (2.45,0.25);	
\draw[black,line width=0.3mm] plot [smooth] coordinates {(2.35,-0.6) (LBs) (LBd)  (2.35,0.6)} ;
\draw[spring] ($(M1.east) - (0,0.25)$) -- (LBs) 
node [midway,below] { $S_{2}$};
\draw[thick,->] (1.05,-0.4)-- (1.85,-0.05) ;

\draw[damper] ($(M1.east) + (0,0.25)$) -- (LBd)
node [midway,above] {\raisebox{0.1cm}{$c_2$}};

\node[above](dots) at ($(LBs) + (0.3,0.125)$) {...};

\coordinate (RBs) at (2.65,-0.25);	
\coordinate (RBd) at (2.85,0.25);	
\draw[black,line width=0.3mm] plot [smooth] coordinates {(2.75,-0.6) (RBs) (RBd)  (2.75,0.6)} ;

\node[draw,outer sep=0pt,thick] (M2) at (4.9,0) [minimum width=1.5cm, minimum height=1cm] {} node[above] at (4.9,0){$\qquad m$};

\draw[spring] (RBs)--($(M2.west) - (0,0.25)$)   
node [midway,below] {\small $S_{N}$};
\draw[thick,->] (2.95,-0.4)-- (3.75,-0.05) ;
\draw[damper] (RBd) -- ($(M2.west) + (0,0.25)$) 
node [midway,above] {\raisebox{0.1cm}{$c_{N}$}};

\draw[thick] ($(M2.south) + (-0.5,-0.125)$) circle (0.125);
\draw[thick] ($(M2.south) + (0.5,-0.125)$) circle (0.125);
\draw[-triangle 45] ($(M2) - (0.5,0)$) -- ($(M2) + (0.5,0)$)  node[midway,below]{$f_N$};	
\draw[-latex] ($(M2.north) + (-0.5,0.1)$) -- ($(M2.north) + (0.5,0.1)$) node[midway,above]{$q_N$};

\node[pattern=north east lines, pattern color=black] at (7.275,-0.125) (RW) [minimum width=0.25cm,  minimum height=1.75cm] {};

\draw[spring] ($(M2.east) - (0,0.25)$) -- ($(RW.west) - (0,0.125)$) node [midway,below] {  $S_{N+1}$};
\draw[thick,->] (5.95,-0.4)-- (6.75,-0.05) ;
\draw[damper] ($(M2.east) + (0,0.25)$) -- ($(RW.west) + (0,0.375)$)
node [midway,above] {\raisebox{0.1cm}{$c_{N+1}$}};

\draw (-2.25,0.75)--(-2.25,-0.75)--(7.15,-0.75)--(7.15,0.75);

\node[pattern=north east lines, pattern color=black] at (2.45,-0.875) (BW) [minimum width=9.4cm,  minimum height=0.25cm] {};

\end{tikzpicture}

		\end{center}
		\caption{Oscillator chain in Example~\ref{ex:osci}.}
		\label{fig:Osci_chain}
	\end{figure}
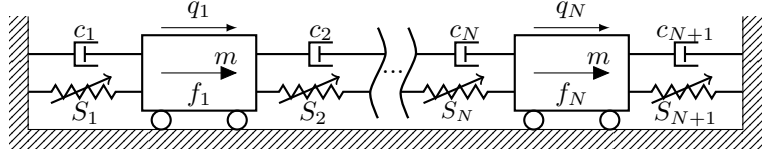
	 
	 The systems investigated by Shaw and Pierre~\cite{SHAW+Pierre}, Breunung and Haller~\cite{TB_Backbone}, Jain et al.~\cite{NPO} satisfy the conditions~\eqref{eq:pars_chain}, as the stiffness forces are of the form~$S_j(\delta)\!=\!k_j \delta+\kappa_j \delta^3$ with~$k_j\!>\!0$ and~$\kappa_j\!\geq\!0$. Therefore, we can guarantee the existence of the  periodic response of these systems for arbitrary large forcing amplitudes. 
	 
	 In the derivations in Appendix~\ref{app:chain_sys}, we further detail that the assumptions on the parameters~\eqref{eq:pars_chain} can be relaxed to include either the cases~$c_{N+1}\!=\!0$ and~$S_{N+1}(q_{N+1})\!=\!0$ or ~$c_1\!=\!0$ and~$S_1(q_1)\!=\!0$. In both cases, the damping matrix~$\mathbf{C}$ and the second derivative of the potential remain positive definite. The conditions on the first and the~$N+1$-th damping coefficient and stiffness force cannot be relaxed simultaneously.	In the case of $S_1(q_1)=S_{N+1}(q_N)\!=\!0$, the system is not connected to the walls and hence a non-periodic, free rigid body motion of the whole chain can be initiated with appropriate forcing. For $c_1\!=\!c_{N+1}\!=\!0$, this motion is undamped and hence if the springs are linear, then forcing at resonance cannot result in a periodic response.

\end{example}

\section{Conclusions}
We have discussed the example of a specific mechanical system for which the application of the harmonic balance procedure leads to the wrong conclusion about the existence of a periodic response. This underlines the necessity of rigorous existence criteria for periodic orbits in damped-forced nonlinear mechanical systems. Such existence criteria can give a priori justification for the use of formal perturbation methods and numerical continuation, eliminating erroneous conclusions or wasted computational resources.  

To obtain such an existence criterion, we have extend a theorem by Rouche and Mawhin~\cite{Rouche_PO} to obtain generally applicable sufficient conditions for the existence of a periodic response in periodically forced, nonlinear mechanical systems. Roughly speaking, these conditions guarantee a periodic orbit under arbitrarily large forcing and response amplitudes, as long as the dissipation acts on all degrees of freedom, the spring forces are potential, and the  potential function is strictly convex or strictly concave outside a neighborhood of the origin. 

Since the conditions of our theorem are sufficient but not necessary, the question arises whether they can be relaxed.  With mechanically relevant examples, we have illustrated that none of the conditions in our theorem can be individually omitted while keeping the others. Based on these results, we identify a large class of nonlinear mechanical systems for which numerical procedures, such as the harmonic balance and the collocation method, are a priori justified. This enables the reliable computation of periodic  orbits for large forcing and oscillation amplitudes in this class of systems.

Theorem~\ref{thm:Existence} guarantees the existence of a periodic orbit but gives not immediate conclusion about the stability of the orbit. For positive definite damping, we do observe both stable and unstable periodic orbits numerically (c.f. Fig~\ref{fig:Duf_pot_neg}) when the conditions of Theorem~\ref{thm:Existence} hold. 

We have limited our discussion to periodic forcing, for which extensive mathematical literature exists. Quasi-periodic forcing is also of interest in engineering applications; indeed, the harmonic balance method has been extended to compute quasi-periodic steady states response of nonlinear mechanical systems (cf. Chua and Ushida~\cite{Chua_QPHB}). The extension of the present results to quasi-periodic forcing, however, is not immediately clear.  

Our discussion is restricted to mechanical equations of motions with position depended nonlinearities, as it is customary in the structural vibrations literature. It is also of interest, however, to extend our conclusions to velocity-dependent nonlinearities.  

\paragraph {Acknowledgements.} {We are thankful to Florian Kogelbauer and Walter Lacarbonara for fruitful discussion on this work. }

\paragraph{Conflict of Interest.} { The authors declare that they have no conflict of interest. }

\paragraph{Funding.}{We recieved no funding for this study.} 

	\setcounter{section}{0}
\renewcommand\thesection{\Alph{section}}
\renewcommand{\thesubsection}{\thesection\arabic{subsection}}
\section{Proofs of the main theorems }
In the following, we prove the main Theorems~\ref{thm:Existence} and~\ref{thm:posdef} and derive an upper bound on the amplitudes of the steady-state responses.
 \subsection{Proof of Theorem~\ref{thm:Existence}}
 \label{app:Exist_proof}
We base our proof of Theorem~\ref{thm:Existence} on a Theorem by Rouche and Mawhin~\cite{Rouche_PO}, who analyze systems of the following form
\begin{equation}
\ddot{\mathbf{q}}+\bar{\mathbf{C}}\dot{\mathbf{q}}+\frac{\partial \bar{V}(\mathbf{q})}{\partial \mathbf{q}}=\mathbf{g}(t), \qquad \mathbf{g}(t)=\mathbf{g}(t+T),~~\bar{V}\in C^1.
\label{eq:RM_sys}
\end{equation}

\begin{theorem}
	\label{thm:RM}
	Assume system~\eqref{eq:RM_sys} satisfies the following conditions: 
	\begin{enumerate}[label=(\subscript{RM}{{\arabic*}})]
		\item The damping matrix $\bar{\mathbf{C}}$ is positive or negative definite.
		\label{cond:damping_RM}
		\item There exists a distance $r\!>\!0$ and an integer $1\!\leq \!n\!\leq N$ such that
		\begin{equation}
		\begin{split}
		q_j\frac{\partial \bar{V}(\mathbf{q})}{\partial q_j} &>0, \qquad |q_j|>r,~~j=1,...,n,
		\\
		q_j\frac{\partial \bar{V}(\mathbf{q})}{\partial q_j}&<0, \qquad |q_j|>r,~~j=n+1,...,N. 
		\end{split}
		\label{eq:sign_cond_RM}
		\end{equation} 
		\label{cond:sign_RM}
		\item The forcing $\mathbf{g}$ is continuous with zero  mean value, i.e., 
		\begin{equation}
		\bar{\mathbf{g}}=\frac{1}{T}\int_{0}^{T}\mathbf{g}(t)dt=\mathbf{0} .
		\end{equation}
		\label{cond:forcing_RM}
	\end{enumerate}
Then system~\eqref{eq:RM_sys} has at least one $T$-periodic solution. 	
\end{theorem}
\begin{proof}
	The proof relies on a homotopy of equation~\eqref{eq:RM_sys} to the equation $\ddot{\mathbf{q}}=0$. Conditions~\ref{cond:damping_RM}-\ref{cond:forcing_RM} ensure a bound on the solution for all homotopy parameters. In addition, condition~\ref{cond:sign_RM} ensures a non-zero Brouwer degree, i.e., the existence of at least one T-periodic solution during the homotopy.  For a detailed proof, we refer to Rouche and Mawhin~\cite{Rouche_PO}.
\end{proof}
We transform system~\eqref{eq:sys0} such that it is in the form~\eqref{eq:RM_sys} and then show that the conditions of Theorem~\ref{thm:Existence} imply that Theorem~\ref{thm:RM} applies. First, we absorb the mean forcing into the potential by setting
\begin{equation}
\tilde{V}(\mathbf{q})=V(\mathbf{q})-\mathbf{q}^T\bar{\mathbf{f}}, \qquad \tilde{\mathbf{f}}=\mathbf{f}-\bar{\mathbf{f}}.
\label{eq:mod_pot}
\end{equation}  
The equation of motion with nonlinearity derived from the potential~$\tilde{V}$ and forcing~$\tilde{\mathbf{f}}$ is equivalent to system~\eqref{eq:sys0}. Further, we right-multiply equation~\eqref{eq:sys0} with the inverse of the mass matrix $\mathbf{M}$ and obtain
\begin{equation}
\ddot{\mathbf{q}}+\mathbf{M}^{-1}\mathbf{C}\dot{\mathbf{q}}+\mathbf{M}^{-1}\frac{\partial \tilde{V}}{\partial \mathbf{q}}=\mathbf{M}^{-1}\mathbf{\mathbf{f}}(t).
\label{eq:eqm_RPO}
\end{equation}
The potential for the geometric nonlinearities of system~\eqref{eq:eqm_RPO} is given by $\bar{V}(\mathbf{q})\!:=\!\tilde{V}(\mathbf{M}^{-1}\mathbf{q})$. Therefore, system~\eqref{eq:eqm_RPO} can be rewritten in the form~\eqref{eq:RM_sys}.  Since the mass matrix is positive definite by assumption, the product $\bar{\mathbf{C}}\!:=\!\mathbf{M}^{-1}\mathbf{C}$ is positive or negative definite. Therefore, condition~\ref{cond:damping} of Theorem~\ref{thm:Existence} implies that condition~\ref{cond:damping_RM} of Theorem~\ref{thm:RM} is satisfied. 

Rewriting condition~\ref{cond:sign_cond} for the the potential $\bar{V}\!:=\!\tilde{V}(\mathbf{M}^{-1}\mathbf{q})$, one recovers the equivalent condition~\ref{cond:sign_RM}. Since~$\tilde{\mathbf{f}}$ has zero mean, condition~\ref{cond:forcing_RM} holds. Therefore, the conditions in Theorem~\ref{thm:Existence} ensure that Theorem~\ref{thm:RM} applies, and hence the existence of a periodic orbit can be guaranteed.

\subsection{Maximal amplitude of the periodic response}
\label{app:bound_ampl}
Essential to the proof of Rouche and Mawhin~\cite{Rouche_PO} is an upper bound on the periodic solution of equation~\eqref{eq:sys0}. In the following, we show that this can be obtained for system~\eqref{eq:sys0} in its original form. Therefore, a transformation to equation~\eqref{eq:RM_sys} is not necessary. We derive an upper bound on the solutions of
\begin{equation}
\mathbf{M} \ddot{\mathbf{q}}+\mathbf{C}\dot{\mathbf{q}}+\mathbf{S}(\mathbf{q})= \mathbf{f}(t),~\qquad \mathbf{q}\in C^2(T),
\label{eq:sys0_lams}
\end{equation}
in the $C^0$ norm defined by 
\begin{equation}
||\mathbf{q}||_{C^0}=\max_{0\leq t\leq T}\left|\mathbf{q}\right|.
\label{eq:def_norm}
\end{equation}
First, we follow the derivation by Rouche and Mawhin~\cite{Rouche_PO} by left-multiplying equation~\eqref{eq:sys0_lams} with ~$\dot{\mathbf{q}}^T$ and integrating over one period to obtain
\begin{equation}
\int_0^T\dot{\mathbf{q}}^T\mathbf{M}\ddot{\mathbf{q}}~dt+ 
\int_0^T\dot{\mathbf{q}}^T\mathbf{C}\dot{\mathbf{q}}~dt
+
 \int_0^T\dot{\mathbf{q}}^T\mathbf{S}(\mathbf{q})~dt
=  \int_0^T\dot{\mathbf{q}}^T\mathbf{f}(t)~dt.
\label{eq:eqm_times_qdot}
\end{equation}
Observing that 
\begin{equation}
\int_0^T\dot{\mathbf{q}}^T\mathbf{M}\ddot{\mathbf{q}}~dt=
\int_0^T \frac{d}{dt}\left(\frac{1}{2}  \dot{\mathbf{q}}^T\mathbf{M}\dot{\mathbf{q}}\right) dt=0,
\end{equation}
where we have used the symmetry of the mass matrix  ($\mathbf{M}=\mathbf{M}^T$)  and the periodicity of~$\mathbf{q}$. Similarly  
\begin{equation}
\int_0^T\dot{\mathbf{q}}^T\mathbf{S}(\mathbf{q})~dt=
\int_0^T\frac{d}{dt}\left(V(\mathbf{q})\right)dt=0,
\end{equation}
where we have used the fact that the geometric nonlinearities arise from a potential~\eqref{eq:potential} and again the periodicity of~$\mathbf{q}$. Therefore, from equation~\eqref{eq:eqm_times_qdot}, we obtain
\begin{equation}
\left|\int_0^T \dot{\mathbf{q}}^T\mathbf{C}\dot{\mathbf{q}} ~dt\right|=\left|\int_0^T\dot{\mathbf{q}}^T \mathbf{f}~ dt \right|.
\label{eq:vel_est}
\end{equation}
With the assumption of of a positive or negative definite~$\mathbf{C}$ matrix (cf. equation~\eqref{eq:damping_cond}), we obtain a lower bound on the left hand side of equation~\eqref{eq:vel_est} to
\begin{equation}
C_0\int_0^T |\dot{\mathbf{q}}|^2 dt\leq \left|\int_0^T \dot{\mathbf{q}}^T\mathbf{C}\dot{\mathbf{q}} dt\right|.
\label{eq:vel_est_LHS}
\end{equation}
For the right hand side of equation~\eqref{eq:vel_est}, we obtain an upper bound by using the Cauchy-Schwartz inequality 
\begin{equation}
\left|\int_0^T\dot{\mathbf{q}} \mathbf{f} ~dt\right|\leq
\left(\int_0^T |\dot{\mathbf{q}} |^2 ~dt\right)^{1/2} \left(\int_0^T |\mathbf{f} |^2 ~dt\right)^{1/2}.
\label{eq:vel_est_RHS}
\end{equation}
Using the definition~\eqref{eq:def_CF} of~$C_f$ and combining the estimates~\eqref{eq:vel_est_LHS} and~\eqref{eq:vel_est_RHS}, we obtain from equation~\eqref{eq:vel_est}, that
\begin{equation}
\left(\int_0^T |\dot{\mathbf{q}} |^2 ~dt\right)^{1/2}\leq \frac{C_f}{C_0}.
\label{eq:ydot_sqrt}
\end{equation}
Equation~\eqref{eq:ydot_sqrt} is and upper bound on the $L_2$-norm of the velocity of the periodic orbit. Rouche and Mawhin~\cite{Rouche_PO} derive the same bound. Now we depart from the derivations by Rouche and Mawhin~\cite{Rouche_PO} and integrate system~\eqref{eq:sys0_lams} for one period, which yields
\begin{equation}
\int_0^T \mathbf{S}(\mathbf{q})~dt=T\bar{\mathbf{f}},\quad \Leftrightarrow \quad \int_0^T \left(\mathbf{S}(\mathbf{q}) -\bar{\mathbf{f}}\right)dt=0,
\label{eq:eqm_intT}
\end{equation}  
where we used the definition~\eqref{eq:f_mean} of the mean forcing~$\bar{\mathbf{f}}$. Applying the mean-value theorem to~\eqref{eq:eqm_intT}, we conclude that there exist $t_j$ such that 
\begin{equation}
S_j(\mathbf{q}(t_j))-\bar{f}_j=0,\qquad 0\leq t_j\leq T,\qquad j=1,...,N.
\label{eq:eqm_MVTHM}
\end{equation}
From condition~\eqref{eq:sign_cond}, we conclude that equation~\eqref{eq:eqm_MVTHM} is only satisfied if~$|q_j(t_j)|<r$. We conclude
\begin{equation}
\begin{split}
q_j^2(t)&=\left(q_j(t_j)+\int_{t_j}^t\dot{q}_j(s)ds\right)^2=
q_j(t_j)^2+2q_j(t_j)\int_{t_j}^t\dot{q}_j(s)ds+\left(\int_{t_j}^t\dot{q}_j(s)ds\right)^2
\\
&\leq r^2+2r|t-t_j|^{\frac{1}{2}}\left(\int_{t_j}^t\dot{q}_j^2(s)ds\right)^{1/2}+
|t-t_j|\int_{t_j}^t\dot{q}_j^2(s)ds
\\&\leq r^2+2r\sqrt{T}\left(\int_{0}^T\dot{q}_j^2(s)ds\right)^{1/2}+
T\int_{0}^T\dot{q}_j^2(s)ds\leq \left(r+\sqrt{T}\frac{C_f}{C_0}\right)^2,
\label{eq:qj^2}
\end{split}
\end{equation}
where we have used the upper bound~\eqref{eq:ydot_sqrt}.  Therefore, for the~$C^0$-norm of the positions,  we obtain 
\begin{equation}
||\mathbf{q}||_{C^0} < \left(\sum_{j=1}^N \sup_{0\leq t \leq T} (q_j^2(t)) \right)^{1/2}\leq \sqrt{N}\left(r+\sqrt{T} \frac{C_f}{C_0} \right) ,
\label{eq:pos_C0_bound}
\end{equation}
 In contrast, Rouche and Mawhin~\cite{Rouche_PO} use the bound~\eqref{eq:ydot_sqrt} to obtain an upper estimate  on the oscillatory part of the position $\tilde{\mathbf{q}}:=\mathbf{q}-1/T\int_0^T\mathbf{q}\,dt$ to $||\tilde{\mathbf{q}}||_{C^0}\leq TC_f/C_0$. From equation~\eqref{eq:eqm_intT}, they directly obtain that each component of the mean $\bar{\mathbf{q}}:=1/T\int_0^T\mathbf{q}\,dt$ is bounded by $r+TC_f/C_0$. Adding the mean and oscillatory part, Rouche and Mawhin~\cite{Rouche_PO} derive the bound 
\begin{equation}
||\mathbf{q}||_{C^0} \leq \sqrt{N}\left(r+\sqrt{T} \frac{C_f}{C_0} \right) +\frac{TC_f}{C_0},
\label{eq:pos_C0_RM}
\end{equation}
which includes the additional summand $TC_f/C_0$ compared to our bound~\eqref{eq:pos_C0_bound}. 

\subsection{Proof of Theorem~\ref{thm:posdef}}
\label{app:convex}
 In the following, we show that condition~\ref{cond:conv_cond} implies that condition~\ref{cond:sign_cond} is satisfied. We note that each continuous function $S_j(\mathbf{q})$ has a maximum and a minimum value in a ball of radius~$r^*$, which we label with $S_{\max}^j$ and $S_{\min}^j$. Choosing the radius 
\begin{equation}
r_j=r^*+\max(0,\frac{\bar{f}_j-S^j_{min}}{C_v},\frac{S_{\max}^j-\bar{f}_j}{C_v}),
\label{eq:r_estimate}
\end{equation}
ensures that the quantity $q_j(S_j(\mathbf{q})-\bar{f}_j)$ has a constant, non-zero sign for all  
\begin{equation}
\mathbf{q}\in\mathbb{Q}_j:=\{\mathbf{q}\in\mathbb{R}^N|~|q_j|> r_j \}.
\end{equation}
First, we assume a positive definite Hessian outside a ball of radius $r^*$. Using a Taylor series expansion of the nonlinearity, we note that outside the $r^*$ ball the following holds:
\begin{equation}
\mathbf{h}^T(\mathbf{S}(\mathbf{q}+\mathbf{h})-\mathbf{S}(\mathbf{q}))
=\int_0^1 \mathbf{h}^T  \frac{\partial^2 V (\mathbf{q}+s\mathbf{h})}{\partial \mathbf{q}^2}\mathbf{h} ds>C_v|\mathbf{h}|^2,\quad  \mathbf{q},\mathbf{h}\in \mathbb{R}^N,\quad 0\leq s \leq 1,\quad|\mathbf{q}+s\mathbf{h}|>r^*.
\label{eq:pos_def_V}
\end{equation}
For every point $\mathbf{q}\!\in\!\mathbb{Q}_j$, we select $\mathbf{h}$ to be the vector pointing from the $q_j$-axis to $\mathbf{q}$ with minimal length.  Denoting the $j$-th unit vector by $\mathbf{e}_j$, we set $\mathbf{h}\!=\!\mathbf{q}-q_j\mathbf{e}_j$ and $\mathbf{q}=q_j\mathbf{e}_j$. Since $|q_j|\!>\!r^*$,  line connecting $q_j\mathbf{e}_j$ and $\mathbf{q}$ is in the region, where the potential $V(\mathbf{q})$ is positive definite. From equation~\eqref{eq:pos_def_V} we obtain 
\begin{equation}
(\mathbf{q}-q_j\mathbf{e}_j)^T(\mathbf{S}(\mathbf{q})-\mathbf{S}(q_j\mathbf{e}_j))=
\sum_{\begin{array}{c}
	n=1
	\\
	n\neq j
	\end{array}}^N
q_n\left[S_n(\mathbf{q})-S_n(q_j\mathbf{e}_j)\right]>0.
\label{eq:r_est_notej}
\end{equation}
Further, we reduce the $j$-th coordinate until we reach $|q_j|=r^*$. The line connecting between the points $\sign(q_j)r^*\mathbf{e}_j$ and $q_j\mathbf{e}_j$ lies in the region with a positive definite Hessian. We evaluate~\eqref{eq:pos_def_V} for $\mathbf{q}=\sign(q_j)r^*\mathbf{e}_j$ and $\mathbf{h}\!=\!\mathbf{q}-\sign(q_j)r^*\mathbf{e}_j$ to obtain
\begin{equation}
\begin{split}
&(\mathbf{q}-\sign(q_j)r^*\mathbf{e}_j)^T(\mathbf{S}(\mathbf{q})-\mathbf{S}(\sign(q_j)r^*\mathbf{e}_j)\\ =
&\sum_{\begin{array}{c}
	n=1
	\\
	n\neq j
	\end{array}}^N
q_n\left[S_n(\mathbf{q})-S_n(q_j\mathbf{e}_j)\right]+
(q_j-\sign(q_j)r^*)\left[S_j(\mathbf{q})-S_j(\sign(q_j)r^*\mathbf{e}_j)\right]
\\
&>(q_j-\sign(q_j)r^*)\left[S_j(\mathbf{q})-S_j(\sign(q_j)r^*\mathbf{e}_j)\right]
>C_v(q_j-\sign(q_j)r^*)^2,\qquad \mathbf{q}\in \mathbb{Q}_j.
\label{eq:r_estimate_ej}
\end{split}
\end{equation}
 For $q_j\!>\!0$, equation \eqref{eq:r_estimate} implies that $(q_j-\sign(q_j)r^*)$ is positive. Therefore, we obtain from equation~\eqref{eq:r_estimate_ej} that 
\begin{equation}
S_j(\mathbf{q})>C_v(q_j-r^*)+S_j(r^*\mathbf{e}_j)>
C_v(q_j-r^*)+S^j_{\min}>
C_v(r^*+\frac{\bar{f}_j-S_{\min}^j}{C_v}-r^*)+S_{\min}^j=\bar{f}_j,\qquad q_j>r_j.
\label{eq:r_estimat_qjpos}
\end{equation}
Similarly, for  $q_j\!<\!0$ the quantity  $S(q_j-\sign(q_j)r^*)$ is negative, therefore equation~\eqref{eq:r_estimate_ej} implies
\begin{equation}
S_j(\mathbf{q})<C_v(q_j+r^*)+S_j(-r^*\mathbf{e}_j)<C_v(q_j+r^*)+S_{\max}^j<C_v(-r^*-\frac{S_{\max}^j-\bar{f}_j}{C_v}+r^*)+S_{\max}^j=\bar{f}_j,\qquad q_j^*<-r_j.
\label{eq:r_estimat_qjneg}
\end{equation}
Equations \eqref{eq:r_estimat_qjpos} and \eqref{eq:r_estimat_qjneg} together imply
\begin{equation}
q_j(S_j(\mathbf{q})-\mathbf{f}_j)>0,\qquad \mathbf{q}\in\mathbb{Q}_j,
\label{eq:final_cond}
\end{equation}
which is equivalent to the upper condition~\eqref{eq:sign_cond}, if we set $r:=\max_j(r_j)$.

The same argument can be repeated for potentials having a negative definite Hessian. The sign in equation~\eqref{eq:pos_def_V} changes, therefore one obtains 
\begin{equation}
q_j(S_j(\mathbf{q})-\mathbf{f}_j)<0,\qquad \mathbf{q}\in\mathbb{Q}_j,
\label{eq:final_cond2}
\end{equation}
which is equivalent to the lower condition~\eqref{eq:sign_cond}, if we set $r:=\max_j(r_j)$.

\section{Derivations for specific examples}

\subsection{Necessary bound on the forcing amplitude for system~\eqref{eq:eqm_counter1} }
\label{app:C1}

In the following, we prove a necessary bound on the forcing amplitude~\eqref{eq:trigwave} for the existence of periodic solutions for system~\eqref{eq:eqm_counter1} with parameters~\eqref{eq:pars_counter1}. Specifically, we assume the existence of a twice continuous differentiable periodic orbit~$\mathbf{q}^*$. Transforming system~\eqref{eq:eqm_counter1} to modal coordinates, we obtain
\begin{subequations}
	\begin{equation*}
		q_1^*=x_1+x_2,\qquad	q_2^*=x_1-x_2,
	\end{equation*}
	\begin{equation}
		\ddot{x}_1+c_1\dot{x}_1+k_1 x_1 +2\kappa x_1^2=-2\kappa x_2^2,
		\label{eq:c1_quad_dof}
	\end{equation}
	\begin{equation}
		\ddot{x}_2+c_1\dot{x}_2+(k_1+2k_2) x_2 =f_1.
		\label{eq:c1_lin_dof}
	\end{equation}
\end{subequations}
The equation of motion of the second modal degree-of-freedom~\eqref{eq:c1_lin_dof} is linear and therefore the assumed periodic response of the second degree of freedom~$x_2$ can be obtained analytically:
\begin{equation}
	\begin{split}
		x_2&=\frac{8f_m}{\pi^2}\sum_{k=0}^{\infty} \frac{(-1)^k\sin((2k+1)\Omega t-\varphi_k)}{(2k+1)^2\sqrt{((k_1+2k_2)-(2k+1)^2\Omega^2)^2+((2k+1)c_1\Omega)^2}}=\frac{8f_m}{\pi^2}\sum_{k=0}^{\infty} c_k\sin((2k+1)\Omega t-\varphi_k) ,
		\\
		\varphi_k&=\tan^{-1}\left( \frac{(2k+1)c_1\Omega}{(k_1+2k_2)-(2k+1)^2\Omega^2}\right),
	\end{split}
	\label{eq:x2_sol}
\end{equation}
Here we have relabeled the amplitudes for notational convenience. Next, we integrate~\eqref{eq:c1_quad_dof} over one period and impose periodicity to obtain
\begin{equation}
	\int_0^T  \left(k_1 x_1+2 \kappa x_1^2 \right) dt= -2 \kappa \int_0^T x_2^2 dt,
	= - T \kappa  \frac{64f_m^2}{\pi^4} \sum_{k=0}^{\infty}|c_k|^2.
	\label{eq:c1_mean_val_thm}
\end{equation}
The infinite sum converges to the limit $c_{\infty}$, since it can be majorized by $1/k^6$, i.e.
\begin{equation}
	\begin{split}
		c_{\infty}:=\sum_{k=0}^{\infty}|c_k|^2&=
		\sum_{k=0}^{\infty}\frac{1}{(2k+1)^4(((k_1+2k_2)-(2k+1)^2\Omega^2)^2+((2k+1)c_1\Omega)^2)}
		\\
		&\leq\frac{1}{c_1^2\Omega^2}\sum_{k=0}^{\infty}\frac{1}{(2k+1)^6}\leq\frac{1}{c_1^2\Omega^2}\sum_{k=1}^{\infty}\frac{1}{k^6}. 
	\end{split}
\end{equation}
For the parameters~\eqref{eq:pars_counter1}, we compute the value $c_{\infty}$ numerically and obtain
\begin{equation}
	c_{\infty}:= 1371.7577441>1371.757744027918.
\end{equation}  
By the  mean-value theorem applied to equation~\eqref{eq:c1_mean_val_thm}, there must be a time instance~$t^*\!\in \! \left[ 0,T\right]$ at which the integrand on the left-hand side~multiplied by $T$ is equal to the infinite sum on the right hand side. Calculating the minimum of the parabola  in that integrand  and inserting the numerical parameter values~\eqref{eq:pars_counter1} yields
\begin{equation}
	-\frac{k_1^2}{8\kappa}T\leq \left(k_1\tilde{x}_1(t^*)+2 \kappa \tilde{x}_1^2(t^*) \right)T=-\kappa  T\frac{64f_m^2}{\pi^4}c_{\infty}, \qquad 0\leq t^*<T.
	\label{eq:bound_fm}
\end{equation}
Solving~\eqref{eq:bound_fm} for the forcing amplitude, we obtain
\begin{equation}
	|f_m|<\sqrt{\frac{k_1^2\pi^4}{512 \kappa^2 c_{\infty}}}=0.011777,\qquad \kappa>0.
	\label{eq:f_thres}
\end{equation}
Since the forcing amplitude~\eqref{eq:pars_counter1} is above the threshold~\eqref{eq:f_thres}, the periodic orbit indicated by the harmonic balance method does not exist.

\subsection{Failure of the harmonic balance with infinite harmonics}
\label{app:fejer}
In the following we construct a forcing for the linear system~\eqref{eq:lin_sys}, such that even for infinite number of harmonics in ansatz~\eqref{eq:HB_sol}, the harmonic balance procedure yields a periodic orbit that differs from the actual periodic orbit significantly. Generally speaking, the computability of a finite number of terms in a Fourier series of a periodic solution does not guarantee the pointwise convergence of that series to the periodic orbit. We consider the function 
\begin{equation}
f_f=\sum_{k=1}^{K}\frac{2}{k^2}\sin(p_k t)\sum_{l=1}^{q_k}\frac{1}{l}\sin(l t),\qquad p_k=2^{k^3+1},\quad q_k=2^{k^3},
\label{eq:fejer_forcing}
\end{equation}
which is a truncated version of a classic example due to Fej\'{e}r~(c.f. Edwards~\cite{Edwards}). We note that the function~\eqref{eq:fejer_forcing} is analytic and therefore the forcing 
\begin{equation}
f(t)=\ddot{f}_f+c\dot{f}_f+kf_f,
\label{eq:fejer_forcing2}
\end{equation}
is well-defined. Applying this forcing in system~\eqref{eq:lin_sys}, we obtain the periodic orbit in the form $q^*\!=\!f_f$. The harmonic balance procedure, therefore, produces a Fourier series of the function~\eqref{eq:fejer_forcing}.  As Edwards~\cite{Edwards} details, the function $f_f$ can be bounded from above by a constant independent of $K$, while its Fourier series at $t\!=\!0$ is unbounded for $K\!\rightarrow \!\infty$. Therefore, for large enough $K$ the Fourier series of $f_f$ will deviate from the  function~\eqref{eq:fejer_forcing} at $t\!=\!0$. Choosing an appropriately large $K$ leads to a  large deviation of the approximative periodic orbit obtained by the harmonic balance from the unique periodic orbit of system~\eqref{eq:lin_sys} with forcing~\eqref{eq:fejer_forcing}. Therefore, the harmonic balance fails to approximate the periodic orbit at $t\!=\!0$.

\subsection{Proof of Fact~\ref{thm:PO_tvary} }
\label{app:T_vary}

 In the following, we show that no periodic orbit for system~\eqref{eq:T_vary_stiffness} exists, for an appropriately chosen set of parameters. For these sets of parameters, one of the Floquet multipliers of the unforced limit of system~\eqref{eq:T_vary_stiffness} equals to one in norm. This introduces the possibility of resonance between the external periodic forcing and the nontrivial solution of the homogeneous part~\eqref{eq:mat_unforced}, under which no periodic orbit for system~\eqref{eq:T_vary_stiffness} can exist.
 
 For further analysis, we introduce the matrices and vectors
\begin{equation}
\mathbf{x}:=
\begin{bmatrix}
q_1^*\\
\tilde{q}^*
\end{bmatrix},\quad 
\mathbf{A}(t):=
\begin{bmatrix}
0 & 1\\
-k_1-\frac{\kappa A^2}{2}+\frac{\kappa A^2}{2} \cos(2\Omega t -2 \psi) & -c_1
\end{bmatrix},\quad 
\mathbf{g}(t):=
\begin{bmatrix}
0\\
f_1(t)
\end{bmatrix}.
\label{eq:At_sys_notation}
\end{equation}
With the notation~\eqref{eq:At_sys_notation}, we express system~\eqref{eq:T_vary_stiffness} in first-order form
\begin{equation}
\dot{\mathbf{x}}=\mathbf{A}(t)\mathbf{x}+\mathbf{g}(t),\qquad \mathbf{A}(t+T/2)=\mathbf{A}(t),\qquad T=2\pi/\Omega,
\label{eq:At_sys}
\end{equation}
and denote its homogeneous part by
\begin{equation}
\dot{\mathbf{x}}=\mathbf{A}(t)\mathbf{x}.
\label{eq:At_sys_hom}
\end{equation}
Furthermore, we define the adjoint problem,  
\begin{equation}
\dot{\mathbf{y}}=-\mathbf{A}(t)^T\mathbf{y}.
\label{eq:At_adj}
\end{equation}

To show the non-existence of a periodic orbit of system~\eqref{eq:At_sys}, we use the following Theorem:
\begin{theorem}
	\label{thm:PO_Farkas}
	Assume that system~\eqref{eq:At_sys_hom} has~$k$ linearly independent, nontrivial~$T$-periodic solutions and denote $k$ linearly independent$T$-periodic solutions to the adjoint system~\eqref{eq:At_adj} by~$\tilde{\mathbf{y}}_1$,~$\tilde{\mathbf{y}}_2$,...,~$\tilde{\mathbf{y}}_k$. Then the non-autonomous system~\eqref{eq:At_sys} has a~$T$-periodic solution if and only if the orthogonality conditions
	\begin{equation}
	\int_0^T\tilde{\mathbf{y}}_j^T\mathbf{g}(t)dt=0,\qquad j=1,...,k,
	\label{eq:ortho}
	\end{equation}
	hold.
\end{theorem} 
\begin{proof}
	For a proof, we refer to Farkas~\cite{Farkas_PO}.
\end{proof}

First, we note system~\eqref{eq:At_sys_hom} is periodic with period~$T/2$ (cf. equation~\eqref{eq:At_sys}), where $T$ is determined by the external forcing $f_2$ (cf. equation~\eqref{eq:T_vary_stiffness}). We denote the complex conjugate Floquet multipliers of system~\eqref{eq:At_sys} by $\rho_1$ and~$\rho_2$ and further  obtain from Liouvilles theorem that
\begin{equation}
\rho_1\rho_2=e^{\int_{t_0}^{t_0+T/2} \Tr\left[ \mathbf{A}(s)\right] ds}=e^{-\frac{c_1T}{2}}, \qquad \rho_1=\bar{\rho}_2.
\label{eq:Flo_mult}
\end{equation}
Equation~\eqref{eq:Flo_mult} imply that the Floquet multipliers are located either on the circle with radius~$e^{-c_1T/4}$ (red circle in Fig.~\ref{fig:Flo_mult}) or on the real axis (blue line in Fig.~\ref{fig:Flo_mult}) in the complex plane.

If the forcing~$f_2$ is zero, then the parameter~$A$ in system~\eqref{eq:T_vary_stiffness} is zero and, due to the positive damping value~$c_1$, the trivial solution of system~\eqref{eq:At_sys_hom} stable. Therefore, the Floquet multipliers are located on the red circle in Fig.~\ref{fig:Flo_mult}.  If we observe instability of the trivial solution to eq.~\eqref{eq:At_sys_hom} for some non-zero forcing ($A\!\neq\! 0$), then the Floquet multipliers must have crossed the unit circle in the complex plane. In this critical case, one of the Floquet multipliers is either one or negative one, which we mark with a black square in Fig.~\ref{fig:Flo_mult}. 

\begin{figure}[ht!]
	\begin{center}
\begin{tikzpicture}
\draw[->] (-1,2) -- (5,2) node[above]{$\mbox{Re}(\rho)$};
\draw[->] (2,-0.5) -- (2,4.5)node[left]{$\mbox{Im}(\rho)$};
\draw[red,thick] (2,2) circle (2cm);
\draw[dashed] (2,2) circle (2.25cm);
\draw[blue, thick] (-0.75,2) -- (1.95,2);
\draw[blue, thick] (2.05,2) -- (4.75,2) ;

\filldraw (4.2,1.95) rectangle ++ (0.1,0.1);
\draw (4.25,2) node[below]{\colorbox{white}{$1$}};
\draw (4.25,2)-- (4.25,1.8) ;

\filldraw (-0.3,1.95) rectangle ++ (0.1,0.1);
\draw (-0.35,2) node[below]{\colorbox{white}{$-1$}};
\draw (-0.25,2)-- (-0.25,1.8);

\draw[latex -latex] (2,2) -- (3.41421,3.41421) node [midway, above, sloped]  {$e^{-\frac{c_1T}{4}}$};
\end{tikzpicture}
	\end{center}
	\caption{Locations of the Floquet multipliers of system~\eqref{eq:At_sys_hom} in the complex plane. The two critical cases,~$\rho_1=1$ and~$\rho_1=-1$, are marked with black squares. }
	\label{fig:Flo_mult}
\end{figure}
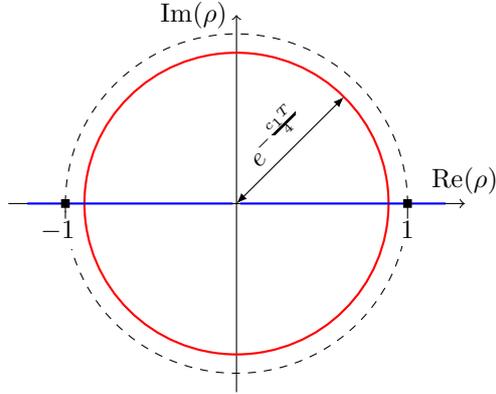

If one of the multipliers,~$\rho_1$, is one, there exists a non trivial~$T/2$-periodic solution of the homogeneous part of system~\eqref{eq:At_sys}. In the case of a Floquet multiplier of negative one, a non-trivial~$T$-periodic solution exists (cf. Farkas~\cite{Farkas_PO}). As Farkas details further, in both cases, the adjoint system~\eqref{eq:At_sys_hom} has a non-trivial~$T/2$ or~$T$-periodic solution, which we denote by~$\tilde{\mathbf{y}}$. Analyzing equation~\eqref{eq:At_adj}, we conclude that a nontrivial~$\tilde{\mathbf{y}}$ implies a non-constant value of both coordinates~$\tilde{y}_1(t)$ and~$\tilde{y}_2(t)$.  We choose  the forcing
\begin{equation}
f_1(t)=
\begin{cases}
\tilde{y}_2(t),& \text{if } \int_0^T\tilde{y}_2(t) dt<0,
\\
-\tilde{y}_2(t)   & \text{if } \int_0^T\tilde{y}_2(t) dt \geq 0,
\end{cases}
\label{eq:f1_forcing}
\end{equation}
which satisfies the negative mean-forcing requirement~\eqref{eq:conv_cond_eval}. Further, note that the forcing~\eqref{eq:f1_forcing}  is~$T/2$ periodic for the case~$\rho_1=1$ and~$T$ periodic in the case~$\rho_1=-1$.  Then the orthogonality condition is 
\begin{equation}
\int_0^{T} \tilde{\mathbf{y}}\mathbf{g}(t)dt=  \pm \int_0^{T} \tilde{y}_2^2dt\neq 0,
\label{eq:orto_cond}
\end{equation}  
where the sign depends on the choice of the forcing~\eqref{eq:conv_cond_eval}. Clearly, the orthogonality condition~\eqref{eq:orto_cond} is not satisfied and therefore, by theorem~\ref{thm:PO_Farkas}  system~\eqref{eq:T_vary_stiffness}, has no periodic solution. 

\subsection{Proof of Fact~\ref{thm:global_min}}
\label{app:global_extrm}
In the following, we show  that no periodic orbit for system~\eqref{eq:sys0} exists if the geometric nonlinearities possess a global minimum, and the mean forcing is below this minimum value (i.e., equation~\eqref{eq:global_min} is satisfied). To prove the nonexistence of a $T$-periodic orbit, we proceed as in Appendix~\ref{app:C1}, assuming the existence of a twice differentiable  periodic orbit~$\mathbf{q}^*$ for system~\eqref{eq:sys0}. Integrating equation~\eqref{eq:sys0} for one period and imposing periodicity yields
\begin{equation}
\int_0^T \mathbf{S}( \mathbf{q}^*(t)) dt=T\bar{\mathbf{f}},\qquad \Leftrightarrow \qquad \int_0^T \left( \mathbf{S}(\mathbf{q}^*(t)) -\bar{\mathbf{f}} \right)dt=0 .
\label{eq:c2_int}
\end{equation}   
By the mean-value theorem, there exist time instances $t^*_j$ within the period at which the integrand in equation~\eqref{eq:c2_int} is equal to zero, i.e.,
\begin{equation}
S_j(\mathbf{q}^*(t^*_j))-\bar{f}_j=0,\qquad j=1,...,N,\quad 0\leq t_j\leq T.
\label{eq:mvt_S}
\end{equation}
 However, due to the choice of the forcing~\eqref{eq:global_min}, we obtain for $j\!=\!l$ that
\begin{equation}
S_l(\mathbf{q}^*(t^*))-\bar{f}_l>0,
\end{equation}   
which contradicts~\eqref{eq:mvt_S}. Therefore, the periodic orbit cannot exist.

\subsection{Proof of Fact~\ref{thm:xsprt_NL}}
\label{app:xsprt_NL}
In the following, we prove that if the forcing amplitude $f$ in the oscillator~\eqref{eq:eqm_x^2} is above the threshold~\eqref{eq:c4_f_thres}, then no periodic solution to system~\eqref{eq:eqm_x^2} exists. Again, we assume the existence of a twice continuous differentiable periodic orbit~$q^*$ and split the coordinate~$q^*$ into a constant and a purely oscillatory part, i.e.
\begin{equation}
\bar{q}:=\frac{1}{T}\int_0^Tq^*(t)dt,\qquad  \tilde{q}(t)=q^*-\bar{q}.
\label{eq:q_split}
\end{equation}	
Substituting the definitions~\eqref{eq:q_split} into the equation of motion~\eqref{eq:eqm_x^2}, yields
\begin{equation}
\ddot{\tilde{q}}+c\dot{q}+\omega_0^2(\bar{q}+\tilde{q})+\kappa (\bar{q}^2+2\bar{q}\tilde{q}+\tilde{q}^2) =f\cos(\Omega t).
\label{eq:eqm_qt}
\end{equation}
Integrating equation~\eqref{eq:eqm_qt} over one period, we obtain
\begin{equation}
\int_0^T \tilde{q}^2dt =-T\left(\frac{\omega^2}{\kappa}\bar{q}+\bar{q}^2\right)\leq\frac{T\omega^4}{4\kappa^2},
\label{eq:qt_L2_bnd}
\end{equation}
where we have used that~$\tilde{q}$ has zero mean (cf. definition~\eqref{eq:q_split}). Furthermore, we note that the left-hand side of~\eqref{eq:qt_L2_bnd} is positive. Since the right-hand side of equation~\eqref{eq:qt_L2_bnd} is a parabola which is concave downwards, it is positive on a closed interval. We thus obtain the upper bound on~$\bar{q}$ in the form 
\begin{equation}
|\bar{q}|<\frac{\omega^2}{|\kappa|},
\label{eq:x0_bound}
\end{equation} 
which is independent of the sign of~$\kappa$. Since~$q^*$ is twice continuously differentiable, it can be expressed in a convergent Fourier series. We denote the Fourier coefficients of~$\tilde{q}$ by
\begin{equation}
\tilde{q}^k:=\frac{1}{T}\int_0^Tq_t e^{-\mathrm{i} k \Omega t} dt,~\qquad k\in\mathbb{Z}.
\label{eq:xt_Fseries}
\end{equation}  
Using Parseval's identity and equation~\eqref{eq:qt_L2_bnd}, we obtain an upper bound on the Fourier coefficients of the assumed periodic orbit as follows
\begin{equation}
|\tilde{q}^k|\leq\left(\sum_{k\in \mathbb{Z}}|\tilde{q}^k|^2\right)^{1/2}=\left( \frac{1}{T} \int_0^T \tilde{q}^2dt\right)^{1/2}\leq \frac{ \omega^2}{2|\kappa|},\qquad k\in \mathbb{Z}.
\label{eq:xt_bound}
\end{equation}
Multiplying equation~\eqref{eq:eqm_qt} with~$e^{-\mathrm{i}\Omega t}$ and integrating over one period yields
\begin{equation}
\int_0^T(\ddot{\tilde{q}}+c\dot{\tilde{q}}+\omega_0^2\tilde{q}+2\kappa\bar{q}\tilde{q}) e^{-i \Omega t} dt+\int_0^T\kappa \tilde{q}^2 e^{-i \Omega t} dt=\frac{f}{2}.
\label{eq:F_entw}
\end{equation}
From equation~\eqref{eq:F_entw}, we obtain
\begin{equation}
\begin{split}
\left|\frac{f}{2}\right|
&\leq\left|\int_0^T(\ddot{\tilde{q}}+c\dot{\tilde{q}}+\omega_0^2\tilde{q}+2\kappa\bar{q}\tilde{q}) e^{-i \Omega t} dt\right|+|\kappa|\int_0^T |\tilde{q}^2(t)||e^{-i\Omega t}| dt 
\\
&\leq  |(-\Omega^2+\mathrm{i}c\Omega+\omega^2+2\kappa\bar{q})\tilde{q}^1|+|\kappa|\frac{\omega^4}{4\kappa^2}
\leq 
\frac{ \omega^2}{ 2|\kappa|} \left(|-\Omega^2+\mathrm{i}c\Omega+\omega^2|+ 2\omega^2 \right)+|\kappa|\frac{\omega^4}{4\kappa^2},
\end{split}
\label{eq:c4_contracition}
\end{equation}
where we have used the upper bounds~\eqref{eq:x0_bound} and~\eqref{eq:xt_bound}.  Equation~\eqref{eq:c4_contracition} gives an upper bound for the forcing amplitude $f$ of the oscillator~\eqref{eq:eqm_x^2}. For forcing amplitudes exceeding this threshold, we obtain a contradiction and therefore no periodic orbit can exist for the oscillator~\eqref{eq:eqm_x^2}.

\subsection{Proof of Fact~\ref{thm:chain_sys}}
\label{app:chain_sys}
We show that the chain system~\eqref{eq:eqm_chain} with the parameters~\eqref{eq:pars_chain} satisfies the conditions of Theorem~\ref{thm:Existence} and hence a steady-state response exists. First, we show that the conditions~\ref{cond:potential} and~\ref{cond:conv_cond} on the geometric nonlinearities  are satisfied for the set of parameters~\eqref{eq:pars_chain}. The definiteness of the damping matrix (i.e. condition~\ref{cond:damping}) can be shown in a fashion~ similar to the definiteness of the Hessian. 

As for condition~\ref{cond:potential}, the spring forces of system~\eqref{eq:eqm_chain} can be derived from the potential
\begin{equation}
V(\mathbf{q})=\int_0^{q_1}S_{1}(-p) dp +\sum_{j=2}^{N} \int_0^{q_{j-1}-q_j} S_{j}(p) d p+\int_0^{q_N}S_{N+1}(p) dp.
\label{eq:pot_chain}	
\end{equation}
Since the spring forces in of system~\eqref{eq:eqm_chain} are continuous by assumption, the integrals in equation~\eqref{eq:pot_chain} exist. With the notation 
\begin{equation}
S_{j,l}:=\frac{\partial }{\partial q_l}\left(S_j(q_{j-1}-q_j)\right),
\end{equation}
the  Hessian of the potential is given by 
\begin{equation}
\small
\mathbf{H}:=\frac{\partial^2 V(\mathbf{q})}{\partial \mathbf{q}^2}=
\begin{bmatrix}
-S_{1,1}+S_{2,1} & S_{2,2}& 0&& &
\\
 -S_{2,1}& -S_{2,2}+S_{3,2} & S_{3,3}& 0 &&
\\
0& -S_{3,2}& -S_{3,3}+S_{4,3} &S_{4,4}& 0 & 
\\
 &0&\ddots &\ddots&\ddots &0
 \\
 & &0 &-S_{N\!-\!1,N\!-\!2} & -S_{N\!-\!1,N\!-\!1} + S_{N,N\!-\!1} &S_{N,N}
 \\
  &  &   &0& -S_{N,N\!-\!1} & -S_{N,N}+S_{N\!+\!1,N}
\end{bmatrix}.
\label{eq:chain_hessian}
\end{equation}
 Due to the choice of parameters~\eqref{eq:pars_chain}, we have following identities
 \begin{equation}
 S_{j,j}<0,\qquad S_{j+1,j}>0,\qquad S_{j,j}=-S_{j,j-1},
 \end{equation} 
 which implies that  the main diagonal entries of the Hessian~\eqref{eq:chain_hessian} are positive and the off-diagonal elements negative. We define the matrices
 \begin{equation}
 \mathbf{H}^j=\begin{bmatrix}
 -S_{1,1}-S_{2,2} & S_{2,2}& 0&& &
 \\
 S_{2,2}& -S_{2,2}-S_{3,3} & S_{3,3}& 0 &&
 \\
 0& S_{3,3}& -S_{3,3}-S_{4,4} &S_{4,4}& 0 & 
 \\
 &0&\ddots &\ddots&\ddots &S_{j,j}
 \\
 &  &   &0& S_{j,j} & -S_{j,j}
 \end{bmatrix}\in\mathbb{R}^{j\times j},
 \label{eq:Hj}
 \end{equation}
 which are equivalent to the leading minors of the Hessian, except for the last term in the main diagonal where the term $-S_{j+1,j}$ is missing. Therefore, $\mathbf{H}^N$ is not equal to  $\mathbf{H}$. The matrices $\mathbf{H}^j$ can be constructed recursively as follows
 \begin{equation}
 \mathbf{H}^1=-S_{1,1}, 
 \quad 
 \mathbf{H}^{j+1}=
 \begin{bmatrix}
 \mathbf{H}^j & \mathbf{0}
\\
 \mathbf{0}&0
 \end{bmatrix}
 +
 \begin{bmatrix}
 \mathbf{0} & \mathbf{0} &\mathbf{0}
 \\
\mathbf{0} & -S_{j,j} &S_{j,j}
\\
\mathbf{0} & S_{j,j} &-S_{j,j}
 \end{bmatrix}.
  \end{equation}
  We show that the matrices $\mathbf{H}^j$ are positive definite by induction. As a first step, we note that $\mathbf{H}^1$ is positive definite. Performing the induction step, we have
  \begin{equation}
  \mathbf{x}^T\mathbf{H}^{j}\mathbf{x}=
  \mathbf{x}^T \begin{bmatrix}
  \mathbf{H}^{j-1} & \mathbf{0}
  \\
  \mathbf{0}&0
  \end{bmatrix}
   \mathbf{x}+
  \mathbf{x}^T 
  \begin{bmatrix}
  \mathbf{0} & \mathbf{0} &\mathbf{0}
  \\
  \mathbf{0} & -S_{j,j} &S_{j,j}
  \\
  \mathbf{0} & S_{j,j} &-S_{j,j}
  \end{bmatrix}
   \mathbf{x}.
   \label{eq:induc_step}
  \end{equation}
  Since the matrix $\mathbf{H}^{j-1}$ is positive definite, the first summand in~\eqref{eq:induc_step} is always positive unless $\mathbf{x}$ aligns with the $x_j$-axis, i.e.  $x_1\!=\!x_2\!=\!...\!=\!x_{j-1}\!=\!0$. Along this axis the first quadratic form is zero, the second quadratic form, however, yields $-S_{j,j}x_j^2$ which is positive. For the case $x_j\!=\!0$ and $|\tilde{\mathbf{x}}|\!=\!|\left[x_1,\dots,x_j-1\right]^T|\!>\!0$, we obtain
  \begin{equation}
  \tilde{\mathbf{x}}^T\mathbf{H}^{j-1} \tilde{\mathbf{x}}+\mathbf{x}^T 
  \begin{bmatrix}
  \mathbf{0} & \mathbf{0} &\mathbf{0}
  \\
  \mathbf{0} & -S_{j,j} &S_{j,j}
  \\
  \mathbf{0} & S_{j,j} &-S_{j,j}
  \end{bmatrix}
  \mathbf{x}
  \geq \tilde{\mathbf{x}}^T\mathbf{H}^{j-1} \tilde{\mathbf{x}},\quad   |\tilde{\mathbf{x}}|>0,
  \end{equation}
  where we have used the fact, that the matrix in the second quadratic form in equation~\eqref{eq:induc_step} is positive semi definite.  Mering both cases
  \begin{equation}
  \mathbf{x}^T\mathbf{H}^{j}\mathbf{x}\geq
  \left\{
  \begin{array}{l l l}
-S_{j,j} x_j^2 >0, & |\tilde{\mathbf{x}}|=0, & |x_j|>0,
 \\[3mm]
 \tilde{\mathbf{x}}^T\mathbf{H}^{j-1} \tilde{\mathbf{x}} >0, & |\tilde{\mathbf{x}}|>0, & x_j=0,
  \end{array}
  \right.  
  \end{equation}
   which implies positive definiteness of all matrices $\mathbf{H}^j$.   Since the Hessian can be written as the sum of the positive definite matrix $\mathbf{H}^N$ and a positive semidefinite the matrix, i.e.
  \begin{equation}
\mathbf{H}=\mathbf{H}^N+
\begin{bmatrix}
\mathbf{0} & \mathbf{0}\\
\mathbf{0} & -S_{N+1,N}	
\end{bmatrix},
\label{eq:add_last}
  \end{equation}
  we conclude that the Hessian~\eqref{eq:chain_hessian} positive definite. 
  
  Since the damping matrix is in the form of the Hessian~\eqref{eq:chain_hessian}, the positive definiteness proof applies for the damping matrix as well. Therefore, we have verified the remaining condition~\ref{cond:damping} of Theorem~\ref{thm:Existence}, and the existence of a periodic orbit is guaranteed by Theorem~\ref{thm:Existence}.

  We note that in the case of $S_{N+1,N}\!=\!0$, the Hessian $\mathbf{H}$ coincides with the matrix $\mathbf{H}^N$, which is positive definite. Therefore, the assumptions on the parameters~\eqref{eq:pars_chain} can be relaxed to include this case.

\end{document}